\documentclass[]{siamart220329}

\usepackage{comment}

\DeclareMathOperator{\interior}{int}
\DeclareMathOperator{\TG}{TG}
\DeclareMathOperator{\F}{F}

\usepackage{graphicx,wrapfig,eso-pic}
\usepackage{amsmath,amsfonts,amssymb,mathtools,enumitem}
\usepackage{caption}
\usepackage{subcaption}
\usepackage{mathrsfs}
\usepackage{blindtext}
\usepackage{tikz}
\usetikzlibrary{automata, positioning, arrows}

\pgfdeclarelayer{edgelayer}
\pgfdeclarelayer{nodelayer}
\pgfsetlayers{edgelayer,nodelayer,main}

\tikzstyle{none}=[inner sep=0pt]
\definecolor{hexcolor0xf81e1c}{rgb}{0.973,0.118,0.110}
\definecolor{hexcolor0x3c00ff}{rgb}{0.235,0.000,1.000}

\tikzstyle{vertex}=[circle, fill=white,draw=black, scale=0.55]
\tikzstyle{whitevertex}=[circle,fill=white,draw=black, scale = 0.5]
\tikzstyle{redvertex}=[circle,fill=hexcolor0xf81e1c,draw=black, scale = 0.5]
\tikzstyle{bluevertex}=[circle,fill=hexcolor0x3c00ff,draw=black, scale = 0.5]
\tikzstyle{greenvertex}=[circle,fill=green,draw=black, scale=0.5]
\tikzstyle{purplevertex}=[circle,fill=magenta,draw=black, scale=0.5]
\tikzstyle{yellowvertex}=[circle,fill=yellow,draw=black, scale=0.5]
\tikzstyle{grayvertex}=[circle,fill=gray,draw=black, scale=0.5]
\tikzstyle{blackvertex}=[circle,fill=black,draw=black, scale=0.5]
\tikzstyle{cyanvertex}=[circle,fill=cyan,draw=black, scale=0.5]

\tikzstyle{textbox}=[rectangle,fill=none,draw=none]
\tikzstyle{box}=[rectangle,fill=none,draw=black]

\tikzstyle{arc}=[black, ->]
\tikzstyle{grayarc}=[gray, ->]
\tikzstyle{bluearc}=[blue, ->]
\tikzstyle{grayedge}=[draw=gray]
\tikzstyle{blueedge}=[draw=blue, thick]
\tikzstyle{rededge}=[draw=red, thick]
\tikzstyle{purpleedge}=[draw=magenta, very thick]
\tikzstyle{greenedge}=[draw=green, thick]
\tikzstyle{blackedge}=[draw=black, thick]
\tikzstyle{dashededge}=[draw=black, dashed]
\tikzstyle{edge}=[draw=black]

\usetikzlibrary{arrows}
\newcommand{\midarrow}{\tikz \draw[-triangle 90] (0,0) -- +(.1,0);}

\theoremstyle{plain}

\newtheorem{condition}{Condition}
\newtheorem{remark}{Remark}
\renewcommand\refname{}

\title{Entropy bounds for Glass networks}
\author{Benjamin W. Wild\thanks{Dept. of Mathematics and Statistics, University of Victoria, P.O. Box 1700 STN CSC, Victoria, B.C., Canada, V8W 2Y2 (\email{benwwild@uvic.ca}).}\and  Roderick Edwards\thanks{Dept. of Mathematics and Statistics, University of Victoria, P.O. Box 1700 STN CSC, Victoria, B.C., Canada, V8W 2Y2 (\email{edwards@uvic.ca}).}}

\begin{document}
\maketitle

\begin{abstract}
We develop a procedure here to calculate good upper bounds on the entropy of Glass networks (a class of piecewise-linear Ordinary Differential Equations), by means of symbolic representations of the continuous dynamics. Our method improves on a result by Farcot (2006), and allows in principle arbitrary refinements of the estimate, and we show that in the limiting case, these estimates converge to the true entropy of the symbolic system corresponding to the continuous dynamics. As a check on the method, we demonstrate for an example network that our upper bound after only a few refinement steps is very close to the entropy estimated from a long numerical simulation. This procedure could be applied to estimating entropy of free-running electronic circuits built from standard Boolean logic gates, which can be modeled by Glass networks.
\end{abstract}



\section{Introduction} \label{sec:intro}

Glass networks are a class of piecewise-linear systems of ordinary differential equations (ODEs), which were originally proposed as qualitative models of gene regulation~\cite{EdwardsGlass2006,Glass1975a,Glass1975b,Glass1977,Glass1985,GlassEdwards2018,GlassKauffman1973}, but have also been implemented in electronic circuitry~\cite{Farcot2019, Glass2005, Luo2020} where the sharpness of the switching between high and low states makes the model a more accurate approximation of such circuit behaviour than of typical gene networks where switching may not be so steep. Demonstration of chaotic, or at least aperiodic behaviour in low-dimensional examples of such networks has been shown in several ways~\cite{Edwards2000, Edwards2001, Edwards2012, LiYang2006, Mestl1996}, and there is a renewed interest in establishing such complex behaviour because of their potential application as robust True Random Number Generators (TRNGs)~\cite{Farcot2019}.

Entropy is one measure of complexity in a dynamical system~\cite{Adler1965}. It is particularly relevant in the context of TRNGs, since the objective in that situation is to generate sequences of bits with positive entropy. For Glass networks, entropy estimates have been discussed by Glass et al.~\cite{ESAG2001} and more rigorously by Farcot~\cite{Farcot2006}. Farcot obtains an upper bound on entropy of the dynamics of a Glass network, by means of the entropy of a Transition Graph (TG) associated with the equations. The $\TG$ captures the transitions between states (rectangular regions of phase space, called {\em boxes}) that are allowed by the structure of the equations. The entropy of this $\TG$ is shown, using a corresponding symbolic dynamics, to be an upper bound for the actual dynamics of the network. However, it is often not a very sharp upper bound, since the $\TG$ typically allows a great many more transitions than actually occur in a trajectory of the differential equations, especially if the latter is restricted to an attractor ({\em i.e.}, ignoring any transient behaviour), which is the pertinent object for the behaviour of a circuit acting as a TRNG.

Here we show how Farcot's approach to producing an upper bound on entropy for a Glass network can be improved, in some cases dramatically improved, by taking into consideration more information about the actual dynamics of the network. As a first step, the $\TG$ can be pruned to remove transitions that do not occur on the attractor of interest. Periodic attractors in Glass networks can often be identified directly, and proven to exist and to be asymptotically attracting. Of course, they must have zero entropy, but one can also sometimes prove that no such stable periodic orbit exists. A trapping region on a {\em wall} between boxes, serving as a Poincar\'e section, can often be identified, even when there is no stable periodic orbit, which restricts trajectories to a certain subgraph of the $\TG$. Furthermore, it is often possible to refine the entropy estimate further by identifying sequences of cycles on the Poincar\'e section that do not occur, and thus correspond to forbidden blocks in the corresponding symbolic dynamics. We use a particular Glass network as a test example, but the methodology is applicable to the entire class of networks considered.

In Section~\ref{sec:glassnet}, we introduce briefly Glass networks and the analysis of their dynamics, mainly following the notation of Farcot~\cite{Farcot2006}. In Section~\ref{sec:TGentropy} we outline Farcot's proof that the entropy of the $\TG$ is an upper bound for the entropy of the network dynamics. Then in Section~\ref{sec:improved_bound} we adapt this approach to show how improved upper bounds may be obtained. In Section~\ref{sec:example} we apply this analysis to a particular Glass network. Finally, in Section~\ref{sec:numerics} we show by means of numerical simulations how much closer our upper bound is to numerical estimates of entropy than the entropy based on the entire $\TG$.

\section{Glass Networks} \label{sec:glassnet}

\subsection{General theory}\label{sec:glassnet_general}
Glass networks are of the form 
\begin{equation} \label{eq:GlassNet}
\frac{dx}{dt}=-\Lambda x+\Gamma(x)\,,
\end{equation}
where $x\in\mathbb{R}^n$, $\Lambda\in\mathbb{R}^{n\times n}$ is a positive diagonal matrix with diagonal entries $\lambda_{i}$, $i=1,\ldots,n$, and $\Gamma:\mathbb{R}^n\rightarrow\mathbb{R}^n$ is constant on each of a finite number of rectangular regions that partition phase space defined by thresholds, as detailed below. The following background on the structure and analysis of Glass networks closely follows Farcot~\cite{Farcot2006} since this work builds upon what is discussed there.

Since $\Gamma$ takes only a finite number of values, $\Gamma$ is bounded and so is the flow. The dynamics of interest occur within the $n$-dimensional rectangular region $$\mathcal{U}=\prod_{i=1}^n\left[\min_{x\in{\mathbb{R}^n}}\left\{\frac{\Gamma_i(x)}{\lambda_{i}}\right\},\,\,\max_{x\in{\mathbb{R}^n}}\left\{\frac{\Gamma_i(x)}{\lambda_{i}}\right\}\right]\,,$$
which is positively invariant. 
The flow from any (non-negative) initial point outside $\mathcal{U}$ must flow into it and cannot then escape. Note that Farcot~\cite{Farcot2006}, defines $\mathcal{U}$ as $[0,1]^n$, but this would only be true after a normalization that we do not make. Since $\Lambda$ is a constant diagonal matrix, Equation~\eqref{eq:GlassNet} only describes systems with linear degradation rates. Thus, we do not here include systems with coupling in the degradation term. $\mathcal{U}$ is partitioned by the piecewise definition of the coupling term, $\Gamma$, which is constant on the interior of {\it boxes}, $n$-dimensional rectangular regions that are products of bounded intervals. Thus, in each box Equation~\eqref{eq:GlassNet} is a first order linear system. 

The set of boundary values are denoted by $\Theta_i=\{\theta_{i,j}\,|\,j\in \{0,...,p_i\}\}$ for each coordinate $x_i$ of $x$. Like Farcot~\cite{Farcot2006}, we will assume that the sets $\Theta_i$ are ordered $\theta_{i,0}<\theta_{i,1}<...<\theta_{i,p_i-1}<\theta_{i,p_i}$, where $\theta_{i,0}=\min_{x\in{\mathbb{R}^n}}\left\{{\Gamma_i(x)}/{\lambda_{i}}\right\}$ and $\theta_{i,p_i}=\max_{x\in{\mathbb{R}^n}}\left\{{\Gamma_i(x)}/{\lambda_{i}}\right\}$, and the set of {\it threshold} values is $\{\theta_{i,j}\,|\,j\in \{1,...,p_i-1\}\}$.
Boxes can then be written as 
\begin{equation} \label{eq:boxes}
B_a=B_{a_1...a_n}=\prod_{i=1}^n[\theta_{i,a_i},\theta_{i,a_i+1}]\,,
\end{equation}
where $a$ belongs to the finite set
\begin{equation} 
\mathcal{A}=\prod_{i=1}^n\{0,...,p_i-1\}\,.
\label{eq:boxalphabet}
\end{equation}
The boundary values define the boundaries of a finite number of boxes in phase space, on the interior of each of which $\Gamma$ is constant. These boundaries in phase space are referred to as {\it walls}. 
Since the subscript $a$ uniquely determines a box, $\mathcal{A}$ defines an alphabet of symbols that will later be used in symbolic dynamical systems that qualitatively describe the dynamics of System~\eqref{eq:GlassNet}. Since $\Gamma$ is constant on the interior of a box, it will also be convenient later to consider the mapping $\Gamma$ as $\Gamma:\mathcal{A}\rightarrow\mathbb{R}^n$. ($\Gamma$ should really be considered as the composition of a map from $\mathbb{R}^n$ to $\cal A$ and a map from $\cal A$ to $\mathbb{R}^n$, but it will not cause confusion to abuse notation and refer to the latter map as $\Gamma$).

Equation~\eqref{eq:GlassNet} can be solved within a given box $B_a$, where $\Gamma$ is a constant vector and the system is first order linear, to give:
\begin{equation}\label{eq:boxflow}
x(t)=f(a)+e^{-\Lambda t}(x(0)-f(a))\,,
\end{equation}
where $f=\Lambda^{-1}\circ \Gamma=\left[\Gamma_1/\lambda_{1},...,\Gamma_n/\lambda_{n}\right]^\top$, and $f(a)$ is referred to as a focal point for the box $B_a$. The flow is clearly attracted towards $f(a)$ (in each variable). If $f(a)\in B_a$, $f(a)$ is an asymptotically stable steady state. Otherwise, before the trajectory can reach $f(a)$, it will intersect a wall between $B_a$ and another box. If the wall is crossed then $\Gamma$ takes a new constant value (in general) and the trajectory continues. The boundary of $B_a$ is formed by $k$-faces that are $k$-dimensional rectangles where $k\in\{0,...,n-1\}$. When the trajectory crosses an $(n-1)$-dimensional face (a wall) there is at most one adjacent box the trajectory can continue in, and so the value of $\Gamma$ is unambiguous, unless the flow is towards the wall on both sides. In the latter case (a {\it black wall}), the wall is not crossed. Rather, the subsequent flow is constrained to the wall, and we need either Filippov theory for discontinuous systems, or singular perturbation theory to determine the flow. If the trajectory crosses a face that has dimension less than $n-1$ there are several possible adjacent boxes that the trajectory could continue into, and again Filippov solutions or singular perturbation is sometimes required. Here, we will restrict the discussion to walls, {\em i.e.}, faces of dimension $n-1$, that are actually crossed. These are called {\it transparent} walls. 

Continuous trajectories from wall crossing to wall crossing can be easily constructed, as long as walls are transparent. In order to ensure this, we will assume that the following two conditions are met.

\begin{condition}\label{cond1}
    $\forall\, a\in\mathcal{A},\: f(a)\in\bigcup_{b\in\mathcal{A}}\interior(B_b).$ 
\end{condition}
Here $f:\mathcal{A}\rightarrow \mathbb{R}^n$ gives the focal point for the box $B_a$, and $\text{int}(B_a)$ is the interior of $B_a$. This assumption ensures that no focal points are located on a box boundary. 

\begin{condition}\label{cond2}
    $\forall\, i\in \{1,...,n\}$, $\forall\, a,a'\in\mathcal{A}$ such that $a-a'=\pm \boldsymbol{e}_i,$ either
    $$\left(\boldsymbol{d}_i(f(a))-a_i\right)\left(\boldsymbol{d}_i(f(a'))-a'_i\right)>0,$$
    or
    $$\left(\boldsymbol{d}_i(f(a))-a_i\right)=0\quad\text{and}\quad\left(\boldsymbol{d}_i(f(a'))-a'_i\right)(a_i-a'_i)>0$$
    or 
    $$\left(\boldsymbol{d}_i(f(a'))-a'_i\right)=0\quad\text{and}\quad\left(\boldsymbol{d}_i(f(a))-a_i\right)(a_i-a'_i)>0\,,$$
    where $\boldsymbol{e}_i$ is the $i^{th}$ basis vector in $\mathbb{R}^n$ and $\boldsymbol{d}=(\boldsymbol{d}_1,...,\boldsymbol{d}_n):\bigcup_a\text{int}(B_a)\rightarrow\mathcal{A}$ is the discretizing operator, which maps a point lying inside a box to the subscript denoting the box.
\end{condition}
This Condition ensures that the flow on both sides of a wall between adjacent boxes is in the same direction, so that black (or white) walls cannot occur. Boxes $a$ and $a'$ are adjacent in the $x_i$ direction when $a-a'=\pm \boldsymbol{e}_i$, and the flow direction in box $a$ of variable $x_i$ is determined by the sign of $\left(\boldsymbol{d}_i(f(a))-a_i\right)$. In the context of gene regulation, Condition~\ref{cond2} is satisfied when there is no direct auto-regulation in the system.

{
Given a box $B_a$ and the solution trajectory defined by Equation~\eqref{eq:boxflow} within the box, if a wall is crossed, the time and position that the trajectory encounters the wall is easy to calculate. Which wall, if any, is encountered uniquely depends on the position of the focal point for the box. The wall $\{x_i=\theta_{i,a_i}\}$ can be crossed if and only if $f_i(a)<\theta_{i,a_i}$ and the wall $\{x_i=\theta_{i,a_i+1}\}$ can be crossed if and only if $f_i(a)>\theta_{i,a_i+1}$. The set of possible output walls, either on the upper or lower side in each direction, is thus captured by 
$$I_{\text{out}}(a)=I_{\text{out}}^+(a)\cup I_{\text{out}}^-(a)\,,$$
where $I_{\text{out}}^+(a)=\{i\in\{1,...,n\}\,|\,f_i(a)>\theta_{i,a_i+1}\}$ and 
$I_{\text{out}}^-(a)=\{i\in\{1,...,n\}\,|\,f_i(a)<\theta_{i,a_i}\}$. Note that if $i\in I_{\text{out}}^+(a)$ then $i\notin I_{\text{out}}^-(a)$ and vice versa. Now for each direction in $I_\text{out}(a)$, the time it takes for $x(t)$ to reach the threshold from an initial point, $x^{(0)}$ according to the flow of $B_a$ is given by
$$\tau_i(x^{(0)})=-\frac{1}{\lambda_i}\ln\left(\frac{f_i(a)-\theta_{i,a_i}}{f_i(a)-x^{(0)}_i})\right)\quad\text{if}\quad i\in I_{\text{out}}^-(a)$$
and
$$\tau_i(x^{(0)})=-\frac{1}{\lambda_i}\ln\left(\frac{f_i(a)-\theta_{i,a_i+1}}{f_i(a)-x^{(0)}_i})\right)\quad\text{if}\quad i\in I_{\text{out}}^+(a)\,.$$
Finally, the time to encounter a wall from $x\in B_a$ is given by
\begin{equation} 
\tau(x^{(0)})=\min_{i\in I_\text{out}(a)}\tau_i(x^{(0)})
\label{eq:boxtime}
\end{equation}
and putting this into Equation~\eqref{eq:boxflow} gives the exit point of $B_a$ from the initial condition $x^{(0)}\in B_a$. The starting point of a trajectory can be chosen to be on a wall, and then a the subsequent sequence of wall transitions can be calculated by repeating the above. It follows from Equation~\eqref{eq:boxflow} and Equation~\eqref{eq:boxtime} that the map from wall to wall is given by $\mathcal{M}_a:\partial B_a\rightarrow\partial B_a$ where 
\begin{equation}\label{eq:boxmap}      
\mathcal{M}_a x =x(\tau(x^{(0)}))=f(a)+e^{-\Lambda\tau(x^{(0)})}(x^{(0)}-f(a))\,.
\end{equation}

Within each box $B_a$, $I_\text{out}(a)$ determines all boxes that are reachable from $B_a$. These  boxes are adjacent to $B_a$ through walls supported by hyperplanes that take the form $\{x_i=\theta_{i,j}\}$, where $i\in I_\text{out}$ and $j\in\{a_i,a_i+1\}$  depends on whether $i$ belongs to $I_\text{out}^-(a)$ or $I_\text{out}^+(a)$. Using that distinction, walls can be denoted as 
$$W_i^+(a)=\{x\,|\,x_i=\theta_{i,a_i+1}\}\cap B_a\quad\text{and}\quad W_i^-(a)=\{x\,|\,x_i=\theta_{i,a_i}\}\cap B_a\,.$$
Each box can then be partitioned into regions associated with each element of $I_\text{out}(a)$ such that only a single adjacent box is reachable from each region. The only walls by which trajectories can escape $B_a$ are then $W_i^+(a)$ for $i\in I_\text{out}^+(a)$, and $W_i^-(a)$ for $i\in I_\text{out}^-(a)$. So, given any initial condition $x$ on a wall, the directions $i$ such that $\tau(x)=\tau_i(x)$ are those for which $\mathcal{M}_ax\in W_i^+(a)\cup W_i^-(a)$. Normally there is only one such direction and a particular wall through which we exit $B_a$, but it is possible that multiple walls are reached simultaneously ({\em i.e.}, we reach an intersection of threshold hyperplanes), in which case there is more than one such direction, $i$. Now $\partial B_a$ can be partitioned into two regions: 
$$\partial B_a^\text{out} =\left(\bigcup_{i\in I_{\text{out}}^+(a)}W_i^+(a) \right)\cup\left(\bigcup_{i\in I_{\text{out}}^-(a)}W_i^-(a) \right)=\{x\in B_a\, |\,\tau(x)=0\}$$
(points on the boundary of $B_a$ from which it takes no time at all to exit, {\em i.e.}, exit walls), and 
$$\partial B_a^\text{in}=\left(\bigcup_{i\in I_{\text{out}}^+(a)}W_i^-(a) \right)\cup\left(\bigcup_{i\in I_{\text{out}}^-(a)}W_i^+(a) \right)\cup\bigcup_{i\notin I_\text{out}(a)}\left(W_i^-(a)\cup W_i^+(a)\right)\,.$$
Thus, the incoming and outgoing regions are unions of walls, which are closed and cover the boundary of $\partial B_a$. It then follows that $\partial B_a^\text{out}\cap\partial B_a^\text{in}\neq\emptyset$ whenever $\partial B_a^\text{out}\neq\emptyset$, since it includes some threshold intersections. Under Condition~\ref{cond1}, it always holds that $\partial B_a^\text{in}\neq\emptyset$ because $i\in I_\text{out}^+(a)\Rightarrow W_i^-(a)\subset\partial B_a^\text{in}$, $i\in I_\text{out}^-(a)\Rightarrow W_i^+(a)\subset\partial B_a^\text{in}$, and $i\notin I_\text{out}(a)\Rightarrow W_i^-(a)\cup W_i^+(a)\subset\partial B_a^\text{in}$. This first partition of the boundary of $B_a$ only allows for a distinction to be made between exit directions and the others. It follows that $\partial B_a^\text{out}=\emptyset$ if and only if $f(a)\in\text{int}(B_a)$. 
Using this partition, it is proven by Farcot~\cite{Farcot2006} that under the assumptions of Condition~\ref{cond1} and provided $\partial B_a^\text{out}\neq\emptyset$, when $\mathcal{M}_a$ is restricted to the domain and range $\partial B_a^\text{in}$ and $\partial B_a^\text{out}$ respectively, $\mathcal{M}_a$ is a homeomorphism. It is pointed out by Farcot that in the case where $\lambda_i = \lambda\,\, \forall i$, all the inequalities that define the domains and ranges of the map $\mathcal{M}_a$ through a box $B_a$ are affine, so that the regions are polytopes. In the examples we investigate later, this will be the case.

Thus far, the transition mapping $\mathcal{M}_a$ has been rigorously defined as a homeomorphism within the confines of a single box. Now we turn our attention to a transition map defined globally on the whole space. Although maps defined within boxes with nonempty outgoing domains are invertible, boxes with no escaping direction are more problematic. It is natural to map the boundary of these types of boxes to a single point whose pre-image is the entire box boundary. As a result, in general a global mapping will not be invertible at all points. This leads us to the consider only forward iterates of the map in the global space. 

In fact, Condition~\ref{cond2} implies that any outgoing wall $W\subset\partial B_b^\text{out}$, for some $b$, is part of $\partial B_a^\text{in}$, for $B_a$ adjacent to $B_b$ at wall $W$, unless $W$ is only a wall for $B_b$, when it lies on the boundary of the whole domain $\mathcal{U}$. However, Condition~\ref{cond1} implies that in this case $W\subset\partial B_b^\text{in}$. Thus,
$$\bigcup_{a\in\mathcal{A}}\partial B_a=\bigcup_{a\in\mathcal{A}}\partial B_a^\text{in}\,.$$
So any point on $\bigcup_{a\in\mathcal{A}}\partial B_a$ belongs to $\partial B_b^\text{in}$ for some $b\in\mathcal{A}$. If $\partial B_a^\text{out}\ne \emptyset$, then $\mathcal{M}_a$ is well defined by Equation~\eqref{eq:boxmap}, but $\partial B_a^\text{out}= \emptyset$ occurs when $f(a)$ is in the interior of $B_a$. In this case, $f(a)$ is an asymptotically stable steady state, and all points in $\partial B_a$ are in its basin of attraction, so we can define that $\forall x\in\partial B_a$, $\mathcal{M}_ax=f(a)$ and $\mathcal{M}_af(a)=f(a)$. Then $\{f(a)\}$ has to be added to the domain of $\mathcal{M}_a$. 

It is now convenient to introduce the subset of terminal subscripts 
$$\mathcal{T}=\{a\in\mathcal{A}\,|\,f(a)\in \text{int}(B_a)\}=\{a\in\mathcal{A}\,|\,\boldsymbol{d}(f(a))=a\}$$
and one can then define local transition maps in all boxes as
$$\mathcal{M}_a: x\in\text{Dom}(\mathcal{M}_a)\mapsto\begin{cases}
f(a)+e^{-\Lambda\tau(x)}((x-f(a))\ & \text{if}\hfill\ a\in\mathcal{A}\setminus\mathcal{T}\\
f(a)\hfill & \text{if}\ a\in\mathcal{T}\,.
\end{cases}$$
It then follows that the domain $\text{Dom}(\mathcal{M}_a)=\partial B_a^\text{in}$ for $a\in\mathcal{A}\setminus\mathcal{T}$ and $\text{Dom}(\mathcal{M}_a)=\partial B_a^\text{in}\cup\{f(a)\}$ for $a\in\mathcal{T}$, and globally,
$$\bigcup_{a\in\mathcal{A}}\text{Dom}(\mathcal{M}_a)=\bigcup_{a\in\mathcal{A}}\partial B_a\cup\bigcup_{a\in\mathcal{T}}\{f(a)\}\,.$$

A global mapping may still not properly be defined by the above on all of  $\bigcup_{a\in\mathcal{A}}\text{Dom}(\mathcal{M}_a)$. The problem arises where an $x\in\text{Dom}(\mathcal{M}_a)$ maps to the intersection of two or more walls, in which case the choice of local map is not unique, and in some cases the subsequent flow depends on this choice. To avoid this issue, following Farcot~\cite{Farcot2006}, we exclude all codimension-2 faces from the analysis, along with all points from which those faces can be reached. On such a domain, a global map can be well defined. Farcot expresses the global map in terms of all the local maps for each box $B_a$ using the indicator function ($\boldsymbol{1}_C(x)= 1$ for $x\in C$, 0 otherwise) as, 
\begin{equation}\label{eq:map}
\mathcal{M}x=\sum_{a\in\mathcal{A}}\boldsymbol{1}_{\text{Dom}(\mathcal{M}_a)}(x)\mathcal{M}_ax\,,
\end{equation}
but defines its domain as 
\begin{equation}\label{eq:dom}
\mathscr{D}_0=\bigcup_{a\in\mathcal{A}}\text{Dom}(\mathcal{M}_a)\setminus\bigcup_{k\in\mathbb{N}}\mathcal{M}^{-k}(\mathscr{F}_2)\,,
\end{equation}
where $\mathscr{F}_2$ is the union of all threshold faces of codimension 2 or more, 
$\mathcal{M}^{k}$ is the $k^{th}$ iterate of $\mathcal{M}$, and $\mathcal{M}^{-k}(\mathscr{F}_2)$ is the $k^{th}$ pre-image of the set $\mathscr{F}_2$.

\subsection{Equal decay rates and single thresholds}\label{sec:equaldecay}

Now, as Farcot~\cite{Farcot2006} points out,
$\left(\mathscr{D}_0,\mathcal{M}\right)$ is a properly defined one-sided discrete dynamical system, whose orbits are  $\{\mathcal{M}^kx\}_{k\in\mathbb{N}}$, for $x\in\mathscr{D}_0$. The iterates of $\mathcal{M}$ are compositions of local maps, as determined by the sequence of walls crossed. The general form of these map compositions are derived by Farcot, but for the purpose of explicit calculations on specific examples, we will need to impose another condition on the networks:
\begin{condition}\label{cond3}
    $\forall\, i,\, \lambda_i=\lambda$ and each variable has only a single threshold value.
\end{condition}
A single threshold in each variable implies that in Equation~\eqref{eq:boxalphabet}, $p_i=2$ for all $i$, and all thresholds may be translated to $0$, by $y_i=x_i-\theta_i$ for each $i$. 

Under the assumptions of Condition~\ref{cond3}, the analysis follows that of Edwards~\cite{Edwards2000}, and the local maps are all fractional linear:
\begin{equation}\label{eq:flm}
\mathcal{M}_ay=\frac{\boldsymbol{\beta}_a y}{1+ \psi_a^{\top} y} \,,
\end{equation}
where 
\begin{equation} \boldsymbol{\beta}_a=I-\frac{f(a)\boldsymbol{e}_j^{\top}}{f(a)^{\top}\boldsymbol{e}_j}\,,\quad \psi_a=\frac{-\boldsymbol{e}_j}{f(a)^{\top}\boldsymbol{e}_j}\,,\label{eq:Bpsi}
\end{equation}
and $\boldsymbol{e}_j$ is again the $j^{th}$ standard basis vector, and $j$ is the index of the exit wall of $\mathcal{M}_a$, starting from $x$. Then compositions of these fractional linear maps are also fractional linear, and denoting $\mathcal{M}_{a_k}, \boldsymbol{\beta}_{a_k}$, and $\psi_{a_k}$ more simply as $\mathcal{M}^{(k)}, \boldsymbol{\beta}^{(k)}$, and $\psi^{(k)}$, after $m$ steps the composite map is 
\begin{equation} \mathcal{M}^{(m-1)}\cdots \mathcal{M}^{(0)}y=\frac{\boldsymbol{\beta}^{(m,0)}y}{1+\psi^{(m,0)^{\top}} y }\,,\label{eq:flmtraj}
\end{equation}
where 
$$\boldsymbol{\beta}^{(m,0)}=\boldsymbol{\beta}^{(m-1)}\ldots \boldsymbol{\beta}^{(1)}\boldsymbol{\beta}^{(0)}\,,\; \mbox{ and }\; \psi^{(m,0)} =\psi^{(0)}+\sum_{k=1}^{m-1}\boldsymbol{\beta}^{(k,0)^{\top}}\psi^{(k)}\,.$$
Thus, for a cycle, where after $m$ steps the trajectory returns to its initial wall (we may say $W^{(m)}=W^{(0)}$, using the same numbering for walls as for the maps), we have 
\begin{equation}\label{eq:cyclemap}
\mathcal{M}:W^{(0)}\rightarrow W^{(0)},\quad \mathcal{M}y=\frac{\boldsymbol{\beta}y}{1+\psi^{\top}y}\,,
\end{equation}
with $\boldsymbol{\beta}=\boldsymbol{\beta}^{(m,0)}$ and $\psi=\psi^{(m,0)}$.

A cycle map, as in Equation~\eqref{eq:cyclemap}, is in general not defined on an entire wall, but on a subset of a wall, called a returning cone: $C=\{y\in W^{(0)}\, |\, \mathcal{M}(y)\in W^{(0)}\}$. As detailed in Section~\ref{sec:improved_bound}, one way of computing a cycle's returning cone is to follow the starting wall backwards through the cycle using pre-images of $\mathcal{M}$ and intersect these with each wall along the cycle. In practice, this means identifying which region on a wall maps to the correct region in the next wall. Outside this region, a different wall (not the one following the cycle) will be reached at the next or a subsequent step. Thus, whenever an entry wall to a box along the cycle maps to multiple exit walls, an inequality must be satisfied on the entry wall for each alternative exit variable to ensure that such a diversion from the cycle does not occur. These inequalities must then be mapped back from the relevant entry wall to the starting wall.
The inequality to prevent an alternative exit on the $k^{th}$ wall of the cycle is~\cite{Edwards2000}
\begin{equation}\label{eq:altexit}
-\frac{\boldsymbol{e_i}^\top}{f_i^{(k)}}\boldsymbol{\beta}^{(k)} y^{(k)}>0,
\end{equation}
where $i$ represents an alternative exit direction, and the denominator of the map \eqref{eq:flm}, which is necessarily positive, has been omitted. Note that there may be more than one such $i$ for a given wall. 

The corresponding inequality on the starting wall is
\begin{equation}\label{eq:altexitstart}
-\frac{\boldsymbol{e_i}^\top}{f_i^{(k)}}\boldsymbol{\beta}^{(k)}\boldsymbol{\beta}^{(k-1)}\ldots \boldsymbol{\beta}^{(0)} y^{(0)}>0.
\end{equation}
The region satisfying this inequality on the starting wall gets mapped to the region on the $k^{th}$ wall that satisfies the inequality~\eqref{eq:altexit}. Thus, in order for a trajectory starting on the starting wall to return to the starting wall following the given cycle, it must satisfy the set of inequalities~\eqref{eq:altexitstart} for each alternative exit variable around the cycle. This set of inequalities defines the returning cone: $C=\{ y\in W^{(0)}\;|\; R y\geq 0\}$ where $W^{(0)}$ is the starting boundary and $R$ is a matrix with one row for each alternate exit variable around the cycle, given by
\begin{equation}\label{eq:retconerow}
R_i=-\frac{\boldsymbol{e_i}^\top}{f_i^{(k)}}\boldsymbol{\beta}^{(k)}\boldsymbol{\beta}^{(k-1)}\ldots \boldsymbol{\beta}^{(0)}
\end{equation}
as in the left hand side of~\eqref{eq:altexitstart}.

\section{Symbolic Dynamics Approach} \label{sec:TGentropy}
 
As discussed in the previous section, the dynamics of Glass networks can be represented (without loss of information) using a continuous-space, discrete-time dynamical system. This discretization process can be taken a step further to discretize space as well. Glass network dynamics can be qualitatively represented using symbolic dynamical systems and as a result, techniques from symbolic dynamics can be used to draw useful conclusions about the original system. The following is a summary of the construction used by Farcot to analyze the irregularity of Glass network dynamics using symbolic dynamics. For a detailed description of the construction, see Farcot~\cite{Farcot2006}. 

The partitioning of phase space into boxes along with the transition maps between them allows a natural encoding of the allowable dynamics of the system in terms of a directed graph, called a transition graph and denoted $\TG=\left(\mathcal{A},\mathcal{E}\right)$. The $\TG$ vertices ($\mathcal{A}$) are box subscripts and the edges ($\mathcal{E}$) correspond to pairs of adjacent boxes that can be successively crossed by a trajectory. This includes 1-loops, to handle the case of a focal point in the interior of its own box, $f(a)\in\text{int}(B_a)$, and pairs that are adjacent through an $(n-1)$-dimensional wall. By Condition~\ref{cond2}, $B_a$ and $B_b$ are adjacent via a single wall if and only if $a-b=\pm \boldsymbol{e}_i$, for $i\in\{1,...,n\}$. Trajectories traversing a box $B_a$ can only exit through a wall $W^{\pm}_i$ for $i\in I_\text{out}(a)$ with the direction given by $\varepsilon_i=\text{sign}(\boldsymbol{d}_i(f(a))-a_i)$. Thus, the edge set of the $\TG$ is 
$$\mathcal{E}=\left\{(a,a)\,|\,\ a\in\mathcal{T}\right\}\cup\left\{(a,a+\varepsilon_i\boldsymbol{e}_i)\,|\, a\in\mathcal{A}\setminus\mathcal{T},\ i\in I_\text{out}(a)\right\}\,.$$
The $\TG$ describes transitions between boxes that can occur through faces that have dimension $(n-1)$, {\em i.e.}, through interiors of walls. Now, attractors of the continuous-space, discrete time system $\left(\mathscr{D}_0,\mathcal{M}\right)$ have a counterpart in the $\TG$, but the converse does not hold in general since trajectories on the $\TG$ may correspond to no trajectories in the original continuous-space system. Some information has been lost in going to the $\TG$ dynamics. We will take this point up again in the next Section.

The $\TG$ encodes the possible dynamics of $\left(\mathscr{D}_0,\mathcal{M}\right)$ into a subset of infinite words on the alphabet  $\mathcal{A}$. The words are given by infinite paths on the graph and the set of all words is given by $$
\mathscr{J}(\TG)=\{\boldsymbol{a}=(a^t)_{t\in\mathbb{N}}\,|\, \forall t\in\mathbb{N},\ (a^t,a^{t+1})\in\mathcal{E}\}\subset\mathcal{A}^\mathbb{N}\,.
$$
The space $\mathscr{J}(\TG)$ can be endowed with a metric to make it a metric space, on which discrete dynamics can be defined by the shift operator $\sigma:\mathscr{J}(\TG)\rightarrow\mathscr{J}(\TG)$, where $(\sigma(\boldsymbol{a}))^t=a^{t+1}$. This operator is continuous, for example, in the metric
$$\rho(\boldsymbol{a},\boldsymbol{b})=\begin{cases}
0\hfill \text{if}\ \boldsymbol{a}=\boldsymbol{b}\,,\\
2^{-\text{min}\{t|a^t\neq b^t\}}\ \text{if}\ \boldsymbol{a}\neq\boldsymbol{b}\,.
\end{cases}$$
 Additionally, $\mathscr{J}(\TG)$ is compact for the metric $\rho$ and is $\sigma$-invariant. As a result, $\mathscr{J}(\TG)$ is a shift space, where $(\mathscr{J}(\TG),\sigma)$ constitutes a discrete dynamical system. Since orbits of this system are associated with words defined on the alphabet $\mathcal{A}$, whose elements in turn represent subsets of the state space of the initial dynamical system, the trajectories of $(\mathscr{J}(\TG),\sigma)$ represent sets of trajectories in $\left(\mathscr{D}_0,\mathcal{M}\right)$.

Now, $\mathscr{D}_0$ lies in the union of all faces of boxes taken without the boundaries of these faces (as well as points $f(a)$ for $a\in\mathcal{T}$). Any of these open faces is well defined by the two boxes it is a part of, except those on the boundary $\partial \mathcal{U}$, but since this boundary cannot be reached from the rest of $\mathscr{D}_0$, we can ignore it and define $\mathscr{D}=\mathscr{D}_0\setminus\partial\mathcal{U}$. So now $\forall\, x\in\mathscr{D}$, there is either a unique pair $(a,b)$ such that $x\in\partial B_a^\text{out}\cap \partial B_b^\text{in}$, or some $a\in\mathcal{T}$ such that $x = f(a)$. We can now define a map $\Phi:\mathscr{D}\rightarrow\mathcal{E}$  as
 $$
 \Phi(x)=\begin{cases}
 (a,b)\ \text{if}\ x\in\partial B_a^\text{out}\cap \partial B_b^\text{in}\\
 (a,a)\ \text{if}\ x=f(a),\ \text{for}\ a\in\mathcal{T}
 \end{cases}
 $$
so that $\Phi$ maps to edges of the $\TG$ instead of vertices (except when $a\in\mathcal{T}$). Note that $\Phi^{-1}(a,b)$ is the open wall between two adjacent boxes $B_a$ and $B_b$. This then leads us to consider a new shift space obtained from $\mathscr{J}(\TG)$ through the 2-block map $\beta_2$ defined as
$$(\beta_2(\boldsymbol{a}))^t=\left[
 \begin{aligned}
 &a^t\\
& a^{t+1}
 \end{aligned}
 \right]\in\mathcal{E}\,,$$
 where $\mathscr{J}(\TG)^{[2]}=\beta_2(\mathscr{J}(\TG))\subset\mathcal{E}^\mathbb{N}$ is a shift space. The shift operator on $\mathscr{J}(\TG)^{[2]}$ is then denoted as $\sigma_{[2]}$. Since $\beta_2$ is a conjugacy~\cite[p.18]{LindMarcus1995} $(\mathscr{J}(\TG),\sigma)$ and $(\mathscr{J}(\TG)^{[2]},\sigma_{[2]})$ are conjugate dynamical systems. 
 
Now to encode the trajectories from $(\mathscr{D},\mathcal{M})$ there are two steps. First, define the mapping $\xi:\mathscr{D}\rightarrow\mathscr{D}^\mathbb{N}$ by $\xi(x)=\left(x,\mathcal{M}x,\mathcal{M}^2x,...\right)$. Since $\mathcal{M}$ is continuous on $\mathscr{D}$ it can be proven that $\xi$ is a conjugacy when its range is restricted to $\xi(\mathscr{D})$ (see Farcot~\cite{Farcot2006} for details). The second step is  the mapping $\Phi_\infty:\mathscr{D}^\mathbb{N}\rightarrow\mathscr{J}(\TG)^{[2]}$, given by 
 $$
 \Phi_\infty\left((x^k)_{k\in\mathbb{N}}\right)=\left(\Phi(x^k)\right)_{k\in\mathbb{N}}\,.
 $$
 Here $\Phi_\infty$ maps sequences in $\mathscr{D}$ to sequences in $\mathcal{E}$. Finally, the map $\phi$ is defined as 
 $$
 \phi=\Phi_\infty \circ\xi:\mathscr{D}\rightarrow\mathscr{J}(\TG)^{[2]}\,.
 $$
 It is shown by Farcot~\cite{Farcot2006} that the mapping $\phi$ takes constant values on the domains 
$$
D_a=\bigcap_{j\in\mathbb{N}}\mathcal{M}^{-j}\left(\Phi^{-1}(a^j,a^{j+1})\right)\,,
$$
which are exactly the connected components of $\mathscr{D}$. Hence, $\phi$ is continuous. 
 
As discussed by Farcot~\cite{Farcot2006}, $\Phi_\infty$, and thus $\phi$, is neither injective nor surjective in general. The non-injectivity of $\phi$ is an inevitable feature of the system $(\mathscr{D},\mathcal{M})$, in which the domains $D_a$ associated with admissible itineraries are not single points. The fact that $\phi$ is not surjective means that some infinite paths on the transition graph do not correspond to any admissible trajectory of the continuous-space system. The $\TG$ represents transitions between boxes that are individually feasible, but there is no guarantee that a finite sequence of such transitions is feasible. Thus, $\phi(\mathscr{D})\subset\mathscr{J}(\TG)^{[2]}$ is a proper subset, and exactly the space of admissible trajectories in $\TG$. As detailed by Farcot~\cite{Farcot2006}, $\phi(\mathscr{D})$ inherits shift-invariance from $\mathscr{D}$'s $\mathcal{M}$-invariance. Now, in order for the shift space $\phi(\mathscr{D})$ to define a dynamical system when paired with $\sigma_{[2]}$, it must be compact~\cite[p.185]{LindMarcus1995}. Since $\mathscr{J}(\TG)^{[2]}$ is compact and contains $\phi(\mathscr{D})$ it would suffice to show that $\phi(\mathscr{D})$ is closed. However, it turns out that $\phi(\mathscr{D})$ is not closed in general.

As discussed by Lind and Marcus~\cite[p.180]{LindMarcus1995} there are two equivalent characterizations of shift spaces: they can be defined as shift-invariant compact subspaces of the full shift, or as subspaces of all infinite words on the alphabet defined by a set of forbidden blocks. Now since $\phi$ is constant on each $D_a$, $\phi(\mathscr{D})$ is the set of words $\epsilon=\beta_2(a)$ such that $D_a\neq\emptyset$. Letting $D_a^i=\bigcap_{k=0}^i\mathcal{M}^{-k}\left(\Phi^{-1}(a^k,a^{k+1})\right)$, where $D_a=\bigcap_{i\in\mathbb{N}}D_a^i$ is empty if either one $D_a^i$ is empty or all of them are nonempty but their full intersection is empty. In the latter case, $\beta_2(a)$ is an infinite word that is forbidden in $\phi(\mathscr{D})$ while all its subwords are allowed. Thus, it does not satisfy the first characterization of a shift space. Naturally then the system to consider is $\left(\overline{\phi(\mathscr{D})},\sigma_{[2]}\right)$. This one is a properly defined symbolic dynamical system since the closure $\overline{\phi(\mathscr{D})}$ is clearly compact. 
Now, the dynamics of the two symbolic dynamical systems $\left(\overline{\phi(\mathscr{D})},\sigma_{[2]}\right)$ and $(\mathscr{J}(\TG)^{[2]},\sigma_{[2]})$ may be compared, by means of entropy. 
Entropy is a nonnegative number that is a measure of the complexity of a dynamical system. Nonzero values of entropy are often taken as an indicator of dynamical chaos. 

In symbolic dynamics, entropy is conjugacy invariant and can be computed easily for systems described by directed graphs. This makes it a useful tool for comparing symbolic dynamical systems. We first give the classical definition of entropy for symbolic dynamical systems. 

\begin{definition}
    Let $X$ be a shift space. The entropy of $X$ is defined by 
    \begin{equation}\label{eq:entropy}
        h(X)=\lim_{n\rightarrow\infty}\frac{1}{n}\log|\mathcal{B}_n(X)|\,,
    \end{equation}
    where $|\mathcal{B}_n(X)|$ is the number of blocks of length $n$ for the shift space $X$, and $\log$ is by convention the logarithm with base 2.
\end{definition}
By definition $h(X)$ is nonnegative. A dynamical system with positive entropy is chaotic by some definitions (e.g. DC2~\cite{Downarowicz2014}). In the case when $X$ is defined by way of infinite paths on an oriented graph $G$, let $A$ be the adjacency matrix of $G$: $A_{i,j}\in\{0,1\}$, and $A_{i,j}=1$ if and only if $(i,j)$ is an edge in the graph. Define the irreducible components of $A$ as the equivalence classes for the equivalence relation: $i\sim j$ if $\exists\, p,q\in\mathbb{N},(A^p)_{i,j}\neq0$ and $(A^q)_{j,i}\neq0$. This corresponds exactly to the strongly connected components in $G$. Let $A_i,\ i=1,...,k$ be the sub-matrices of $A$ with all indices in the same equivalence class. If there is a single class, $A$ is said to be irreducible. The Perron-Frobenius theorem ensures that any irreducible matrix with nonnegative entries has a dominant positive eigenvalue $\mu_A$, which is simple, and is associated with a nonnegative eigenvector. Following Lind and Marcus~\cite[p.121]{LindMarcus1995}, the Perron eigenvalue of $A$ is $$\mu_A=\max_{i=1,...,k}\mu_{A_i}$$
where each $\mu_{A_i}$ is the Perron eigenvalue of an irreducible sub-matrix, $A_i$, and the entropy is given by $h(X)=\log\mu_A$.

Finally, since $\mathscr{J}(\TG)$ and $\mathscr{J}(\TG)^{[2]}$ are conjugate and entropy is conjugacy invariant it follows that $h(\mathscr{J}(\TG)^{[2]})=h(\mathscr{J}(\TG))$. As $\mathscr{J}(\TG)$ is exactly the shift space induced by infinite paths on the $\TG$, one simply defines 
$$
h_{\TG}=h(\mathscr{J}(\TG)^{[2]})=h(\mathscr{J}(\TG))\,.
$$
One can then also define the entropy of the true dynamics ($\phi(\mathscr{D})$) of the network as
$$
h_{\phi(\mathscr{D})}=h\left(\overline{\phi(\mathscr{D})}\right)\,.
$$
Since we have that $\overline{\phi(\mathscr{D})}\subset\mathscr{J}(\TG)^{[2]}$, it follows that $h_{\phi(\mathscr{D})}\leq h_{\TG}$.

{

\section{Better Entropy Bounds}\label{sec:improved_bound}

In the previous two sections we discussed the construction from~\cite{Farcot2006} that allows for $(\mathscr{D},\mathcal{M})$ to be mapped into $(\overline{\phi(\mathscr{D})},\sigma_{[2]})$. The entropy of $(\overline{\phi(\mathscr{D})},\sigma_{[2]})$ is then shown to be bounded above by the entropy of the $\TG$, which is an indication of the potential irregularity of the system dynamics. However, for many Glass networks the entire set of dynamics allowed by the $\TG$ may not be realized in the continuous dynamics after transients, so this upper bound on entropy may be very loose. When there exists a {\em trapping region} in phase space, {\em i.e.}, a positively invariant region from which trajectories cannot escape, then the system dynamics may be constrained to a subset of the $\TG$. It may be that all trajectories are eventually confined to one trapping region, but even if there are other attractors, once confined to a trapping region, the entropy of the subsequent dynamics is unaffected by other parts of phase space. Thus, once transients have died out, the long term dynamics can be represented using the subset of the $\TG$ that represents only the trapping region, and the corresponding entropy provides a tighter upper bound than the full $\TG$. 

Additionally it may be that even within the trapping region some dynamics allowed by the $\TG$ subset are also never realized. To account for this, procedures are developed to obtain tighter upper bounds on entropy by eliminating cycles and concatenations of cycles on the $\TG$ that are shown to be impossible or transient in computations of returning regions and their images. The key idea is to represent dynamics within the trapping region by more detailed graphs that disallow the forbidden dynamics still allowed by previous graphs. In doing so, a sequence of progressively more detailed graphs can be constructed where each graph captures more closely the actual dynamics allowed by the trapping region. This sequence can then be shown to have entropies converging to the true entropy within the trapping region. For this section, Condition~\ref{cond3} is not necessary and thus will not be imposed. 

\subsection{Trapping regions}

In order to improve the upper bound on entropy using information about a trapping region, we must first define trapping regions. To do this, a few definitions will be needed, starting with {\em cycles}. Based on the conventions of Lind and Marcus~\cite[p.38]{LindMarcus1995}, 
\begin{definition}
    A path $\pi = e_1\dots e_m$ on a graph $G$ is a finite sequence of edges $e_i$ from $G$ such that $t(e_i) = i(e_{i+1})$ for $1\leq i\leq m-1$, where $i(e)$ and $t(e)$ here indicate the initial and terminal nodes of an edge. The length of $\pi$ is $|\pi|=m$, the number of edges it traverses. The path $\pi$ starts at vertex $i(\pi) = i(e_1)$ and terminates at vertex $t(\pi) = t(e_m)$, and $\pi$ is a path from $i(\pi)$ to $t(\pi)$. A cycle is a path that starts and terminates at the same vertex: $i(\pi)=t(\pi)$. If, in addition, $e_i\ne e_1$ for $i=2,\ldots, m$, we call $\pi$ a first-return cycle (to edge $e_1$).
\end{definition}
Since each cycle on a network's $\TG$ represents a potential sequence of wall transitions in $(\mathscr{D},\mathcal{M})$, for a given cycle $\pi = e_1\dots e_m$, the edge $e_1$ represents a wall from which the path in phase space starts. We will denote this wall as the {\em starting wall} and $e_1$ as the {\em starting edge}. These terms may be used interchangeably if it is clear which context is being referenced. 

In order for a trajectory passing through a starting wall to follow a cycle and return back to that wall, the trajectory must be confined to the region in each wall that maps through the rest of the walls in the cycle. To specify the region of phase space occupied by trajectories following the cycle, or at least its intersections with each wall around the cycle,
we first define the {\em returning region} on the starting wall.
\begin{definition}
Let $\pi=e_1\dots e_m$ be a cycle with starting wall $W_1$ and walls associated with each edge $e_i$ denoted $W_i,\, i\in\{1,\ldots,m\}$. The returning region of $\pi$ on wall $W_1$ is the subset of $W_1$ that under $m$ applications of the map $\mathcal{M}$ will return back to $W_1$ following the path $\pi$. Letting $R_i(S)=\mathcal{M}^{-1}(S)\cap W_i$, this subset is defined as 
\begin{equation}
\mathscr{R}_\pi(W_1)=R_1\circ\cdots\circ R_m(W_1)\,.
\end{equation}
\end{definition}
We can also define the cycle map as follows:
\begin{definition}
    The cycle map, $\mathcal{M}_\pi:\mathscr{R}_\pi(W_1) \to W_1$ is defined by 
    $$
        \mathcal{M}_\pi(x)=\mathcal{M}^m(x), \quad x\in \mathscr{R}_\pi(W_1),
    $$
where $\mathcal{M}^m$ is the $m^{th}$ iterate of $\mathcal{M}$, and $|\pi| = m$.
\end{definition}
When all decay rates in the Glass network are equal (but not in general), these returning regions are polygonal cones commonly referred to as {\it returning cones}. The subset of $\mathscr{D}$ through which trajectories starting in this returning region $\mathscr{R}_\pi(W_1)$ pass is now defined as follows:
\begin{definition}
For a first-return cycle $\pi$ with starting wall $W_1$, the subset of the domain $\mathscr{D}$ traversed by trajectories following $\pi$ is 
\begin{equation}
\mathscr{D}_\pi=\mathscr{D}\cap\bigcup_{i=0}^{|\pi|-1}\mathcal{M}^i(\mathscr{R}_\pi(W_1))\,,
\end{equation}
where $\mathcal{M}^0=id$. We call this the cycle tube associated with cycle $\pi$.
\end{definition}
Using these definitions we can now define a trapping region for a Glass network. 
\begin{definition}
For a given Glass network, and a given starting wall, $W_1$, a trapping region in wall $W_1$ is defined as 
\begin{equation}\label{eq:trapreg}
    \mathscr{T} = \bigcup_{\pi\in\mathbb{T}}\mathscr{R}_\pi(W_1)\,,
\end{equation}
where $\mathbb{T}$ is a set of first-return cycles (possibly infinitely many) on the $\TG$ such that every cycle in $\mathbb{T}$ starts at $e_1$ (the starting edge) and 
$$
    \bigcup_{\pi\in\mathbb{T}}\mathcal{M}_\pi(\mathscr{R}_\pi(W_1))\subseteq
    \mathscr{T}\,.
$$
Its corresponding trapping tube is defined as 
\begin{equation}\label{eq:traptube}
\mathscr{D}_{TR}=\bigcup_{\pi\in\mathbb{T}}\mathscr{D}_{\pi}\,.
\end{equation}
\end{definition}

We note here that transients can often be removed to obtain a smaller trapping region before proceeding with the analysis. In the context of our trapping region and trapping tube, a transient region is any returning region that is non-empty, but after an initial cycle back to the starting wall, is never visited again. This is the case when there are nonempty returning regions that are not mapped into by any $\pi\in \mathbb{T}$. If the returning region of a transient cycle is removed from the trapping region, one is left with a smaller trapping region. Transient regions can be defined as follows.
\begin{definition}\label{def:transientcycle}
    For a given Glass network with a trapping region as defined above, a cycle $\pi_j\in\mathbb{T}$ is a transient cycle (sequence of box transitions) if
    $$
    \bigcup_{\pi\in\mathbb{T}}\mathcal{M}_\pi\left(\mathscr{R}_{\pi}(W_1)\right)\cap\mathscr{R}_{\pi_j}(W_1)=\emptyset\,,
    $$
    and its returning region is a transient region on the wall $W_1$. 
\end{definition}

\begin{definition}
    Let $\mathbb{T}$ be a set of first return cycles that forms a trapping region. Define the set of all cycles in $\mathbb{T}$ with empty returning regions as
    \begin{equation}
        F_1 = \{\pi\in\mathbb{T}\;|\:\mathscr{R}_\pi=\emptyset\}
    \end{equation}
    and define the set of all cycles from $\mathbb{T}$ with non-empty returning regions that are transient under Definition~\ref{def:transientcycle} as 
    \begin{equation}
        \Pi_1 = \{\pi\in\mathbb{T}\;|\; \mathscr{R}_\pi(W_1)\neq\emptyset\;\text{and}\; \pi\; \text{is transient}\;\}. 
    \end{equation}
\end{definition}

We can obtain a trapping region with transients removed for a given Glass network and a given starting wall, $W_1$, as follows. Given a trapping region for the set of cycles $\mathbb{T}$ as given by~\eqref{eq:trapreg}, if the set of transient cycles $\Pi_1\subset \mathbb{T}$ is non-empty then the smaller trapping region is 
\begin{equation}
\bigcup_{\pi\in\mathbb{T}\setminus\Pi_1}\mathscr{R}_\pi(W_1),
\end{equation} 
since
$$
        \bigcup_{\pi\in\mathbb{T}}\mathcal{M}_{\pi}(\mathscr{R}_{\pi}(W_1))\subseteq
        \bigcup_{\pi\in\mathbb{T}\setminus\Pi_1}\mathscr{R}_\pi(W_1).
$$
Its trapping tube is then
\begin{equation}
    \mathscr{D}_{TR}=\bigcup_{\pi\in\mathbb{T}\setminus\Pi_1}\mathscr{D}_{\pi}.
\end{equation}

We are interested in Glass networks in which such a trapping region occurs. To ensure that we are considering only relevant systems we now impose another condition:  
\begin{condition}\label{cond4}
There exists a wall such that some subset $\mathbb{T}\subseteq\Pi$ forms a trapping region, where $\Pi$ is the set of all first-return cycles starting at that wall. 
\end{condition}
This condition means that there is a starting wall that contains a trapping region, so that at least some trajectories always return to the starting wall, with a graph corresponding to the trapping region that is potentially a proper subgraph of the $\TG$. 

We will further assume that we have a minimal trapping region, in the sense that it is not the union of disjoint sets that are themselves trapping regions. Multiple trapping regions correspond to different attractors, each with their own entropy, and we focus here on the entropy of a single attractor. 

In order to represent dynamics, we define the subset of the $\TG$ that corresponds to the trapping region in the following way: 

\begin{definition}
    For a given Glass network satisfying Conditions~\ref{cond1}, \ref{cond2}, and \ref{cond4}, let $\mathbb{T}$ be a set of first-return cycles on the $\TG$ that forms a trapping region. For each $\pi\in\mathbb{T}$, let $G_\pi$ be the subgraph of $\TG$ representing $\pi$. The $\TG_r$ is then defined as $\TG_r = (\mathcal{A}_r,\mathcal{E}_r)$ where
    $$
            \mathcal{A}_r = \bigcup_{\pi\in\mathbb{T}\setminus(F_1\cup\Pi_1)}\mathcal{V}(G_\pi)\,\quad\text{and}\;\quad
            \mathcal{E}_r = \bigcup_{\pi\in\mathbb{T}\setminus(F_1\cup\Pi_1)}\mathcal{E}(G_\pi),
    $$
\end{definition}
and $\mathcal{V}(G_\pi)$ and $\mathcal{E}(G_\pi)$ are the respective vertex and edge sets for the graph $G_\pi$.

As in the previous section, $\TG_r$ encodes the possible dynamics of $\left(\mathscr{D}_\text{TR},\mathcal{M}\right)$ into a subset of infinite words on the alphabet $\mathcal{A}_r$.  The words are similarly given by infinite paths on $\TG_r$ and the set of all words is given by 
$$
\mathscr{J}(\TG_r)=\{\boldsymbol{a}=(a^t)_{t\in\mathbb{N}}\,|\, \forall t\in\mathbb{N},\ (a^t,a^{t+1})\in\mathcal{E}_r\}\subset\mathcal{A}^\mathbb{N}\,.
$$
It should be clear from this definition that $\mathcal{E}_r\subseteq\mathcal{E}$ and $\mathscr{J}(\TG_r)\subseteq\mathscr{J}(\TG)$. In order to compare the continuous dynamics to the dynamics of $\mathscr{J}(\TG_r)$ we must apply the 2-block map to obtain $\mathscr{J}(\TG_r)^{[2]}=\beta_2(\mathscr{J}(\TG_r))\subseteq\mathcal{E}^\mathbb{N}_r\subseteq\mathcal{E}^\mathbb{N}$. It follows that $\mathscr{J}(\TG_r)^{[2]}\subseteq\mathscr{J}(\TG)^{[2]}$. Now we must encode the trajectories of $\left(\mathscr{D}_\text{TR},\mathcal{M}\right)$ in order to compare the  dynamics on the trapping region to $\mathscr{J}(\TG_r)^{[2]}$. The space of admissible trajectories in $\TG_r$ is then exactly $\phi\left(\mathscr{D}_\text{TR}\right)$. We encounter the same problem with $\mathscr{D}_\text{TR}$ as we did with $\mathscr{D}$, namely, that $\phi\left(\mathscr{D}_\text{TR}\right)$ is not compact. So in order to properly define a symbolic dynamical system we must consider the closure, $\overline{\phi\left(\mathscr{D}_\text{TR}\right)}$. It follows that $\overline{\phi(\mathscr{\mathscr{D}_\text{TR}})}\subseteq\mathscr{J}(\TG_r)^{[2]}\subseteq\mathscr{J}(\TG)^{[2]}$. Finally, defining the entropies $h_{\phi(\mathscr{\mathscr{D}_\text{TR}})}=h(\overline{\phi(\mathscr{\mathscr{D}_\text{TR}})})$ and $h_{\TG_r}=h\left(\mathscr{J}(\TG_r)^{[2]}\right)=h\left(\mathscr{J}(\TG_r)\right)$ it follows that $h_{\phi(\mathscr{\mathscr{D}_\text{TR}})}\leq h_{\TG_r}\leq h_{\TG}$. Thus, the entropy of the $\TG_r$ is potentially a better upper bound on the entropy of the true dynamics than is the original $\TG$. 

\subsection{Separated cycle representation of $\TG_r$}\label{sec:splittingTGr}

{

The $\TG_r$ allows for a more accurate representation of the long-term dynamics of a Glass network. However, it may be that it still misses finer details within the trapping region. For example, if multiple cycles on the $\TG_r$ share nodes and edges in common, these can potentially allow additional cycles that have empty returning regions. As a result, while the $\TG_r$ does give an improvement on entropy estimation, it still potentially allows for too much. In order to improve further we need to consider graph representations that separate the allowed cycles from each other in such a way that the forbidden cycles are no longer possibilities. This will in turn produce a more accurate entropy estimate. 

The idea is to produce a new graph in which all the cycles that contribute to the trapping region are separated from each other, and cycles that exist in $\TG_r$ but are not part of the trapping region (because they have empty returning regions) are excluded. In order to do this, however, we will need to impose another condition on the networks that we can consider. 
\begin{condition}\label{cond5}
    There exists a finite $M$ such that $|\pi|\leq M$ for all $\pi\in\mathbb{T}$ where $\mathbb{T}$ is a set of first-return cycles forming a trapping region.  
\end{condition}
This condition avoids the situation where there are an infinite number of distinct first-return cycles whose returning regions form a trapping region and an infinite partition of the starting wall. In our analysis, an infinite number of first-return cycles will lead to infinite alphabets in the derived symbolic dynamical systems, and we prefer not to deal with such difficulties. In most of the examples we have looked at, this situation does not arise, but an example of such a system was presented by Gedeon~\cite{Gedeon2003}.

Under the assumption of Condition~\ref{cond5}, the graph that we propose to use for our new entropy upper bound can be defined as follows:
\begin{definition}
    For a given Glass network satisfying Conditions~\ref{cond1}, \ref{cond2}, \ref{cond4}, and~\ref{cond5}, let $\mathbb{T}$ be the set of first-return cycles that forms a trapping region. The $\TG_r(1)$ is then defined as 
    \begin{equation}
        \TG_r(1) = \left(\bigcup_{\pi\in\mathbb{T}\setminus(F_1\cup\Pi_1)}{\mathcal{A}_\pi},\,\,\mathcal{E}_\text{cross}\cup\bigcup_{\pi\in\mathbb{T}\setminus(F_1\cup\Pi_1)}{\mathcal{E}_\pi}\right),
    \end{equation}
    where $$\mathcal{A}_\pi =\{a_\pi\,|\,a\in\mathcal{V}(G_\pi) \},$$ $$\mathcal{E}_\pi =\{(a_\pi,b_\pi)\,|\,(a,b)\in\mathcal{E}(G_\pi) \},$$  
    $$
        \mathcal{E}_\text{cross} =\bigcup_{\pi,\tau\in\mathbb{T}\setminus(F_1\cup\Pi_1) }\{ (i(e_1)_\pi,t(e_1)_\tau)\,|\,\pi\neq\tau \},
    $$
    and the starting edge, $e_1$, is the same for all $\pi\in\mathbb{T}$. 
\end{definition}
The subscripts $\pi$ and $\tau$ above now distinguish nodes and edges that belong to different cycles. Thus, if an edge $(a,b)$ in $TG_r$ belongs to both cycle $\pi$ and cycle $\tau$, then the graph $\TG_r(1)$ has edges $(a_{\pi},b_{\pi})$ and $(a_{\tau},b_{\tau})$ with distinct vertices.

From the definition, it should be clear that the $\TG_r(1)$ representation does indeed include each cycle contributing to the trapping region and does separate them from each other, while allowing any cycle to follow any other by means of the cross-cycle edges represented by the set $\mathcal{E}_\text{cross}$. 

To show that this new representation does improve on the entropy bound achieved by the $\TG_r$, we will need to define a sliding block code from $\mathscr{J}(\TG_r(1))$ to $\mathscr{J}(\TG_r)$ know as an embedding~\cite[p.18]{LindMarcus1995}. For two shift spaces $X$ and $Y$, an embedding is a sliding block code from $X$ to $Y$ that is one-to-one. Embeddings have the convenient property that if $X$ embeds into $Y$, then $h(X)\leq h(Y)$. Thus, if we can define an embedding from $\mathscr{J}(\TG_r(1))$  to $\mathscr{J}(\TG_r)$, we can prove that the $\TG_r(1)$ produces a potentially lower entropy value than the $\TG_r$.

Up to this point we have exclusively dealt with one-sided shift spaces. The notions of embeddings that we need are defined for two-sided shift spaces by Lind and Marcus~\cite{LindMarcus1995} and thus cannot be directly applied to our situation. However, we can use the natural extension of a one-sided shift space, defined as follows~\cite{French2019}:
\begin{definition}
    For a given one-sided shift space $X$, the natural extension $\widehat{X}$ is the space of bi-infinite sequences such that $x\in\widehat{X}$ if and only if every sub-word of $x$ is an element in the language of $X$. 
\end{definition}
The $\TG_r$ has the convenient property that for any two nodes, there exists a path connecting them. As a result, the $\TG_r$ language, $\mathcal{L}(\mathscr{J}(\TG_r))$, is the same as that of its natural extension $\widehat{\mathscr{J}(\TG_r)}$, {\em i.e.}, $\mathcal{L}(\mathscr{J}(\TG_r)) = \mathcal{L}(\widehat{\mathscr{J}(\TG_r)})$, and therefore
$$
    h(\mathscr{J}(\TG_r)) = h(\widehat{\mathscr{J}(\TG_r)}).
$$
For the purposes of our entropy upper bounds, we can use the natural extension to get the entropy relations we need. 

Now we can define an embedding, as discussed above, by taking each subscripted symbol (node) and removing the subscript to recover a symbol in $\mathscr{J}(\TG_r)$, as follows: 

\begin{definition}  
    Let $\mathbb{T}$ be the set of first-return cycles making up the trapping region. For any $x\in\widehat{\mathscr{J}(\TG_r(1))}$ there exists a $y\in\widehat{\mathscr{J}(\TG_r)}$ such that $x_i=(y_i)_{\pi_i}$ (where $x_i=(x)_i$), for some $\pi_i\in\mathbb{T}$. Define $\psi_1:\widehat{\mathscr{J}(\TG_r(1))}\rightarrow\widehat{\mathscr{J}(\TG_r)}$ as a sliding block code such that  $(\psi_1(x))_i=({y_i})$.
\end{definition}
In simple terms, the map $\psi_1$ takes the subscripted symbols from $\mathscr{J}(\TG_r(1))$ and removes their subscripts. Since the $\TG_r(1)$ is more restrictive than the $\TG_r$, every point from $\mathscr{J}(\TG_r(1))$ has a corresponding point in $\mathscr{J}(\TG_r)$. This map is clearly one-to-one and hence an embedding. Thus, $h(\widehat{\mathscr{J}(\TG_r(1))})\leq h(\widehat{\mathscr{J}(\TG_r)})$. Finally since the $\TG_r(1)$ satisfies the same property as the $\TG_r$ that all nodes can be connected by some path, it follows that 
$$
    h(\mathscr{J}(\TG_r(1)))= h(\widehat{\mathscr{J}(\TG_r(1))})\leq h(\widehat{\mathscr{J}(\TG_r)})=h(\mathscr{J}(\TG_r)).
$$

It is clear that the entropy of $\TG_r(1)$ is less than or equal to the entropy of $\TG_r$. However, it remains to be shown that we can use it to produce an upper bound on entropy for the true dynamics. To demonstrate this is straightforward. Consider the shift of finite type $X_{\mathcal{F}_1}$ defined on the alphabet $\mathcal{A}_r$ with set of forbidden blocks given by
    $$
        \mathcal{F}_1 = F_1\cup\Pi_1\cup\mathcal{F}\,,
    $$
    where $\mathcal{F}$ is the set of forbidden blocks for $\mathscr{J}^{[2]}(\TG_r)$.
From this definition it follows that $\mathcal{F}_1$ is finite and $h_{\phi(\mathscr{D}_{TR})}\leq h(X_{\mathcal{F}_1})$. Then we can define the sliding block code 
$$
    	\zeta_1:\mathscr{J}^{[2]}(\TG_{r}(1))\rightarrow\mathscr{J}^{[2]}(\TG)\,,
	$$
	where $\zeta_1$ is just the mapping that removes all subscripts from symbols, effectively encoding trajectories  into the original alphabet. From the definition of $\zeta_1$, it is clear that it is finite-to-one. Additionally, as detailed by Lind and Marcus~\cite[p.276]{LindMarcus1995}, finite-to-one codes on any shift space preserve entropy. So, it follows that 
	$$
   	 h(\mathscr{J}^{[2]}(\TG_{r}(1))) = h(\zeta_1(\mathscr{J}^{[2]}(\TG_{r}(1)))).
	$$
	Furthermore, it follows that $\zeta_1(\mathscr{J}^{[2]}(\TG_{r}(1))) = X_{\mathcal{F}_1}$, and hence
	$$
    	h(\mathscr{J}^{[2]}(\TG_{r}(1))) = h(\zeta_1(\mathscr{J}^{[2]}(\TG_{r}(1)))) = h(X_{\mathcal{F}_1}).
	$$
Thus, 
$h_{\phi(\mathscr{D}_{TR})}\leq h(X_{\mathcal{F}_1}) = h(\mathscr{J}(\TG_r(1)))\leq h_{\TG_r}\leq h_{\TG}$.

\subsection{Arbitrary levels of refinement}

In the previous subsection, we showed how to remove all cycles not included in $\mathbb{T}$ using the separated cycle representation of the $\TG_r$, the $\TG_r(1)$. However, there may also be concatenations of cycles that are forbidden. It is not uncommon in Glass networks that there are cycles $\pi$ and $\tau$ such that $M_\pi(\mathscr{D}_\pi)\cap \mathscr{D}_\tau = \emptyset$ and hence, cycle $\pi\tau$ is forbidden from ever occurring in the true dynamics. It is easy to then modify the $\TG_{r}(1)$ to reflect this, by simply removing the cross edge between the nodes $i(e_1)_\pi$ and $t(e_1)_\tau$.

Of course, longer concatenations of cycles may also be forbidden, in general. To deal with this, we will need to further generalize the separated cycle representation of a trapping region, beyond returning once to the starting wall as was initially considered. This will allow us to determine concatenations of first-return cycles of length $k\in\mathbb{N}$ that are forbidden.  A general way to identify forbidden concatenations of cycles is to calculate their returning regions. Empty returning regions correspond to forbidden cycles. Returning regions for concatenations of cycles, like $\pi\pi$, $\pi\tau$, $\tau\pi$, and $\tau\tau$, can be calculated in the same way as was done for cycles $\pi$ and $\tau$. We define the tube of a concatenation of first-return cycles as follows: 

\begin{definition}
    For a concatenation of first-return cycles $\pi_1\dots \pi_k$, the concatenation tube is given by
    \begin{equation}
        \mathscr{D}_{\pi_1\dots \pi_k}=\mathscr{D}\cap\bigcup_{i=0}^{|\pi_1|-1}\mathcal{M}^i(\mathscr{R}_{\pi_1\dots \pi_k}(W_1)).
    \end{equation}
\end{definition}
Note that the concatenation tube follows the walls only until the first return ({\em i.e.}, after cycle $\pi_1$).

The returning region of a concatenation of first-return cycles starting with cycle $\pi$ represents a subset of the returning region for cycle $\pi$. The set of all possible concatenations of a fixed number of cycles, $k$, represents a partitioning of the original trapping region as defined previously. Using this, we can redefine the trapping tube $\mathscr{D}_{TR}$ using concatenation tubes in such a way that we can distinguish between regions associated with each concatenation. Thus, for concatenations of length $k$
\begin{equation}
    \mathscr{D}_{TR} = \bigcup_{\pi_1,\dots,\pi_k\in\mathbb{T}}\mathscr{D}_{\pi_1\dots \pi_k}
\end{equation}

It is true that $\mathscr{D}_{TR}$ is partitioned by all the concatenation tubes of all the first-return cycles in $\mathbb{T}$. However, it may be that some of the concatenations within $\mathscr{D}_{TR}$ are transient. Knowing this, it will be beneficial to distinguish between the possible types of concatenations. 
\begin{definition}\label{def:concatenations}
    Let $\mathbb{T}$ be a set of first-return cycles that forms a trapping region. Define the set of all possible concatenations (of length $k$) from $\mathbb{T}$ as  
    \begin{equation}
        \mathbb{T}_k = \{\pi_1\dots\pi_k\;|\;\pi_i\in\mathbb{T}\;\text{for}\;1\leq i\leq k\},
    \end{equation}
    the set of all concatenations (of length $k$) from $\mathbb{T}$ with empty returning regions as 
    \begin{equation}
        \F_k = \{\pi_1\dots\pi_k\;|\; \mathscr{R}_{\pi_1\dots\pi_k}(W_1)=\emptyset\}, 
    \end{equation}
    and the set of all concatenations of cycles (of length $k$) from $\mathbb{T}$ with non-empty returning regions that are transient under Definition~\ref{def:transientcycle} as 
    \begin{equation}
    \Pi_k=\{\pi_1\dots\pi_k\;|\; \mathscr{R}_{\pi_1\dots\pi_k}(W_1)\neq\emptyset\;\text{and}\;\pi_1\dots\pi_k\;\text{is transient}\}. 
    \end{equation}
\end{definition}

Transient concatenation tubes need to be removed from the domain before we continue. This will result in a new trapping region that is a subset of the original. The new domain is simple enough to define. 
\begin{definition}
    Let $\mathbb{T}$ be the set of first-return cycles that forms a trapping region. The domain with transient concatenations removed, $\mathscr{D}_{TR}(k)$, is defined as
    \begin{equation}
        \mathscr{D}_{TR}(k) = \left(\bigcup_{\pi_1\dots\pi_k\in\mathbb{T}_k}\mathscr{D}_{\pi_1\dots \pi_k}\right)\setminus \bigcup_{\tau\in\Pi_k}\mathscr{D}_\tau,
    \end{equation}
    where $\mathscr{D}_\tau$ is the concatenation tube for the concatenation of first-return cycles, $\tau$.
\end{definition}

We can now define a separated representation of $\TG_r$ that separates cycle concatenations of length $k$ and that is consistent with the underlying dynamics. As before, in order for this to give us the correct entropy relations, we will need to consider the natural extension where appropriate. The graph representation that forbids certain cycle concatenations of length $k$ is defined as follows:

\begin{definition}
    For a given Glass network satisfying Conditions~\ref{cond1}, \ref{cond2}, \ref{cond4}, and~\ref{cond5}, let $\mathbb{T}$ be the set of first-return cycles that together forms a trapping region. The $\TG_r(k)$ is then defined as 
    \begin{equation}
        \TG_r(k) = \left(\bigcup_{\pi_1\dots\pi_k\in\mathbb{T}_k\setminus\left(\F_k\cup\Pi_k\right)}{\mathcal{A}_{\pi_1\dots\pi_k}}\,,\,\,\mathcal{E}_\text{cross}\cup\bigcup_{{\pi_1\dots\pi_k\in\mathbb{T}_k\setminus\left(\F_k\cup\Pi_k\right)}}{\mathcal{E}_{\pi_1\dots\pi_k}}\right),
    \end{equation}
    where $$\mathcal{A}_{\pi_1\dots\pi_k} =\{a_{\pi_1\dots\pi_k}\,|\,a\in\mathcal{V}(G_{\pi_1}) \},$$ $$
        \mathcal{E}_{\pi_1\dots\pi_k} =\{(a_{\pi_1\dots\pi_k},b_{\pi_1\dots\pi_k})\,|\,(a,b)\in\mathcal{E}(G_{\pi_1})\setminus\{ e_1\}\} ,
    $$ and 
    $$
        \mathcal{E}_\text{cross} =\bigcup_{\substack{\pi_1\dots\pi_k\in\mathbb{T}_k\setminus\left(\F_k\cup\Pi_k\right)\\
        \tau_1\dots\tau_k\in\mathbb{T}_k\setminus\left(\F_k\cup\Pi_k\right)}}\{(i(e_1)_{\pi_1\dots\pi_k},t(e_1)_{\tau_1\dots\tau_k})\quad |\quad  \pi_2\dots\pi_k=\tau_1\dots\tau_{k-1}
        \}. 
    $$
\end{definition}
Notice that for a given concatenation of first-return cycles of length $k$, the first $k-1$ letters must match the last $k-1$ letters of the previous cycle. For example, the sequence $\pi\tau\tau\pi$ can only be followed by cycles whose subscript starts with $\tau\tau\pi$. This aligns exactly with iterates  within $\mathscr{D}_{TR}(k)$. 

In this way, the graph allows for individual cycle traversals as all the other $\TG$ representations have, but disallows transient concatenations as well as concatenations with empty returning regions. In effect, this representation allows for exclusion of cycle concatenations of arbitrary length. To demonstrate that the $\TG_r(k)$ does indeed produce an improved entropy bound in the way that we desire will require another embedding, $\psi_k$, similar to $\psi_1$. 

\begin{definition}
    Let $\mathbb{T}$ be the set of first-return cycles making up the trapping region. For any $x\in\widehat{\mathscr{J}(\TG_r(k))}$ there exists a $y\in\widehat{\mathscr{J}(\TG_r(k-1))}$ such that $x_i=(y_i)_{\pi_i}$ (where $x_i=(x)_i$), for some $\pi_i\in\mathbb{T}$. Define $\psi_k:\widehat{\mathscr{J}(\TG_r(k))}\rightarrow\widehat{\mathscr{J}(\TG_r(k-1))}$ as a sliding block code such that  $(\psi_k(x))_i=({y_i})$.
\end{definition}
In simple terms, the map $\psi_k$ takes the subscripted symbols from $\mathscr{J}(\TG_r(k))$ and removes their $k^{th}$ subscripts. Since the $\TG_r(k)$ is more restrictive than the $\TG_r(k-1)$, every point from $\mathscr{J}(\TG_r(k))$ has a corresponding point in $\mathscr{J}(\TG_r(k-1))$. Note that under this definition where $k=1$, we retrieve $\psi_1$ as defined previously. This map is also one-to-one and hence an embedding. Thus, 
$$
    h(\widehat{\mathscr{J}(\TG_r(k))})\leq h(\widehat{\mathscr{J}(\TG_r(k-1))}).
$$ 
Finally since each $\TG_r(k)$ satisfies the same property as the $\TG_r$ that all nodes can be connected by some path, if we let $h(\mathscr{J}^{[2]}(\TG_r(k)))= h_{\TG_r(k)}$ it follows that 
$$
    h_{\TG_{r}(k)} \leq\dots h_{\TG_{r}(2)} \leq h_{\TG_{r}(1)} \leq h_{\TG_r}\leq h_{\TG}.
$$
At each inequality we omit mention of the natural extension. However, as before, the inequality relations come from the fact that the entropy of the $\TG$ one-sided systems and their natural extensions are the same and that in the case of the natural extensions, the embeddings induces the inequality relation.

To verify that the $\TG_r(k)$ produces an upper bound on the true dynamics is similar to the case of the $\TG_r(1)$. Consider the shifts of finite type $X_{\mathcal{F}_k}$ defined on the alphabet $\mathcal{A}_r$ with sets of forbidden blocks given by
    $$
        \mathcal{F}_k = \F_k\cup\Pi_k\cup\mathcal{F}_{k-1}\,,
    $$
where $\mathcal{F}_0 = \mathcal{F}$ (the set of forbidden blocks for $\TG_r$). From this definition it follows that $\mathcal{F}_k$ is finite and $h_{\phi(\mathscr{D}_{TR}(k))}\leq h(X_{\mathcal{F}_k})\leq\dots\leq h(X_{\mathcal{F}_1})$. Then we can define the sliding block code 
$$
    	\zeta_k:\mathscr{J}^{[2]}(\TG_{r}(k))\rightarrow\mathscr{J}^{[2]}(\TG)\,,
	$$
	where $\zeta_k$ is just the mapping that removes all subscripts from symbols. From the definition of $\zeta_k$, it is clear that it is finite-to-one. So, it follows that 
	$$
   	 h(\mathscr{J}^{[2]}(\TG_{r}(k))) = h(\zeta_k(\mathscr{J}^{[2]}(\TG_{r}(k)))).
	$$
	Furthermore, $\zeta_k(\mathscr{J}^{[2]}(\TG_{r}(k))) = X_{\mathcal{F}_k}$, and hence
	$$
    	h(\mathscr{J}^{[2]}(\TG_{r}(k))) = h(\zeta_k(\mathscr{J}^{[2]}(\TG_{r}(k)))) = h(X_{\mathcal{F}_k}).
	$$
Note that in Definition~\ref{def:concatenations} the set $\Pi_k$ is finite. This means that in  $\phi(\mathscr{D}_{TR})$ there are only a finite number of transient words associated with the returning regions of elements from $\Pi_k$. Since a finite number of transient words in a shift space will not contribute to the entropy, removing from $\mathscr{D}_{TR}$ the returning regions associated with elements of $\Pi_k$ will not change entropy. Thus for all $k$, $h_{\phi(\mathscr{\mathscr{D}_{\text{TR}}})}=h_{\phi(\mathscr{\mathscr{D}_{\text{TR}}}(k))}$ and we obtain the string of inequalities 
\begin{proposition}
    \begin{equation}
h_{\phi(\mathscr{\mathscr{D}_{\text{TR}}})}=h_{\phi(\mathscr{\mathscr{D}_{\text{TR}}}(k))}\leq h_{\TG_{r}(k)} \leq\dots h_{\TG_{r}(2)} \leq h_{\TG_{r}(1)} \leq h_{\TG_r}\leq h_{\TG}.
\end{equation}
\end{proposition}
This tells us that we can refine the $\TG$ to remove dynamics associated with forbidden cycle concatenations (those with empty returning regions) of any length we wish, and we can still use its entropy as an upper bound on the entropy of the true dynamics. 

{
Finally, we recall the following proposition proven by Lind and Marcus~\cite[p.123]{LindMarcus1995}:
\begin{proposition}\label{prop446}
    Let $X_1\supseteq X_2\supseteq X_3\ldots $ be shift spaces whose intersection is $X$. Then $h(X_k)\rightarrow h(X)$ as $k\rightarrow\infty$. 
\end{proposition}
This gives all the necessary results to prove the following theorem: 
\begin{theorem}\label{theorem1}
    For a given Glass network, $\lim_{k\rightarrow\infty}h_{\TG_{r}(k)} =h_{\phi(\mathscr{\mathscr{D}_{\text{TR}}})}$.
\end{theorem}
\begin{proof}
    From the definition of each $X_{\mathcal{F}_k}$, it follows that 
    $$
    	\bigcap_{k=1}^\infty X_{\mathcal{F}_k} = \overline{\phi(\mathscr{D}_{TR})}.
$$
So by Proposition~\ref{prop446}, it follows that $\lim_{k\rightarrow\infty}h(X_{\mathcal{F}_k}) = h_{\phi(\mathscr{D}_{\text{TR}})}$. Finally, since $h(X_{\mathcal{F}_k}) = h_{\TG_r(k)}$, it follows that 
$\lim_{k\rightarrow\infty}h_{\TG_{r}(k)} =h_{\phi(\mathscr{\mathscr{D}_{\text{TR}}})}$.
\end{proof}
}

Since it is easy to compute the entropy of $\mathscr{J}^{[2]}(\TG_{r}(k))$ for any $k$ from the adjacency matrix of $TG_r(k)$, Theorem~\ref{theorem1} allows us to find an upper bound arbitrarily close to the true entropy. However, it should be emphasized that we have no way to determine a level of refinement, $k$,  guaranteeing that our estimate is within a specified $\varepsilon$ of the true entropy.

\begin{remark} Since the $\TG_r(k)$ adjacency matrices are sparse, computation of the eigenvalues is reliable for large $k$. This gives us a practical method of estimating the entropy of a network attractor to any arbitrary level of refinement, in principle. 
\end{remark}

}

\section{An Example} \label{sec:example}

We now consider an example Glass network to demonstrate the efficacy of these graph refinements. This example will fall into the class of networks that satisfies Condition~\ref{cond3}. Note that although the theory of the previous section applies in general, computation of returning regions is only practicable when Condition~\ref{cond3} holds. Consider the following Glass network: 
\begin{equation}\label{eq:example}
\begin{aligned}
  \dot{y}_1 &= - y_1 + 2\left(\bar{Y}_3Y_4 + Y_2Y_3\right) - 1 \\
  \dot{y}_2 &= - y_2 + 2\left(Y_1\bar{Y}_3 Y_4+\bar{Y}_1 Y_3 Y_4+\bar{Y}_1\bar{Y}_3\bar{Y}_4\right) - 1 \\
  \dot{y}_3 &= - y_3 + 2\left(\bar{Y}_1 Y_2 + Y_1 Y_4\right) - 1 \\
  \dot{y}_4 &= - y_4 + 2\left(Y_2\bar{Y}_3 + \bar{Y}_1 Y_3\right) - 1
\end{aligned}
\end{equation}
where $Y_i=0$ if $y_i<0$, and $1$ if $y_i>0$. $\bar{Y}_i=1-Y_i$. This network is simple in the sense that all its degradation rates are 1, all the focal points are located at $\pm1$, and each variable has only a single threshold. Its $\TG$ can be represented by the 4-dimensional hypercube in Figure~\ref{GMTG} (see~\cite{Edwards2001}). Projections of example trajectories are shown in Figure~\ref{GNPP}. From the adjacency matrix of the $\TG$, we calculate the approximate entropy (the logarithm of the Perron eigenvalue) to be $h_{\TG}\approx0.873$. 
\begin{figure}
    \centering
    {\Large
    \scalebox{0.55}{
    \begin{tikzpicture}[very thick, every node/.style={sloped,allow upside down},scale=.5]
	\node[inner sep=2pt, circle, draw, fill, label={north west:$0000$}] (0000) at (0,0) {};
    \node[inner sep=2pt, circle, draw, fill, label={north west:$0001$}] (0001) at (8,3) {};
    \node[inner sep=2pt, circle, draw, fill, label={north west:$0010$}] (0010) at (7,9) {};
    \node[inner sep=2pt, circle, draw, fill, label={north east:$0011$}] (0011) at (11,6) {};
    \node[inner sep=2pt, circle, draw, fill, label={north west:$0100$}] (0100) at (0,20) {};
    \node[inner sep=2pt, circle, draw, fill, label={south east:$0101$}] (0101) at (8,11) {};
    \node[inner sep=2pt, circle, draw, fill, label={north west:$0110$}] (0110) at (7,29) {};
    \node[inner sep=2pt, circle, draw, fill, label={north west:$0111$}] (0111) at (11,14) {};
    \node[inner sep=2pt, circle, draw, fill, label={east:$1000$}] (1000) at (20,0) {};
    \node[inner sep=2pt, circle, draw, fill, label={north west:$1001$}] (1001) at (16,3) {};
    \node[inner sep=2pt, circle, draw, fill, label={north west:$1010$}] (1010) at (27,9) {};
    \node[inner sep=2pt, circle, draw, fill, label={north west:$1011$}] (1011) at (19,6) {};
    \node[inner sep=2pt, circle, draw, fill, label={north west:$1100$}] (1100) at (20,20) {};
    \node[inner sep=2pt, circle, draw, fill, label={north west:$1101$}] (1101) at (16,11) {};
    \node[inner sep=2pt, circle, draw, fill, label={north west:$1110$}] (1110) at (27,29) {};
    \node[inner sep=2pt, circle, draw, fill, label={east:$1111$}] (1111) at (19,14) {};
    \draw (1000) --node {\midarrow} (0000);
    \draw (0000) --node {\midarrow} (0100);
    \draw (0010) --node {\midarrow} (0000);
    \draw (1100) --node {\midarrow} (0100);
    \draw (1100) --node {\midarrow} (1000);
    \draw (0110) --node {\midarrow} (0010);
    \draw (0100) --node {\midarrow} (0110);
    \draw (0110) --node {\midarrow} (1110);
    \draw (1110) --node {\midarrow} (1010);
    \draw (1010) --node {\midarrow} (1000);
    \draw (1000) --node {\midarrow} (0000);
    \draw (1010) --node {\midarrow} (0010);
    \draw (1110) --node {\midarrow} (1100);
    \draw (0001) --node {\midarrow} (1001);
    \draw (1001) --node {\midarrow} (1011);
    \draw (1011) --node {\midarrow} (0011);
    \draw (0011) --node {\midarrow} (0001);
    \draw (0101) --node {\midarrow} (0001);
    \draw (1001) --node {\midarrow} (1101);
    \draw (1111) --node {\midarrow} (1011);
    \draw (0011) --node {\midarrow} (0111);
    \draw (0101) --node {\midarrow} (1101);
    \draw (0101) --node {\midarrow} (0111);
    \draw (0111) --node {\midarrow} (1111);
    \draw (1101) --node {\midarrow} (1111);
    \draw (0001) --node {\midarrow} (0000);
    \draw (1001) --node {\midarrow} (1000);
    \draw (0010) --node {\midarrow} (0011);
    \draw (1011) --node {\midarrow} (1010);
    \draw (0100) --node {\midarrow} (0101);
    \draw (0110) --node {\midarrow} (0111);
    \draw (1111) --node {\midarrow} (1110);
    \draw (1100) --node {\midarrow} (1101);
    \end{tikzpicture}
    }
    }
    \caption{TG for example Glass network in Equation~\eqref{eq:example}.}
    \label{GMTG}
\end{figure}

\begin{figure}
    \centering
    \includegraphics[scale=.18]{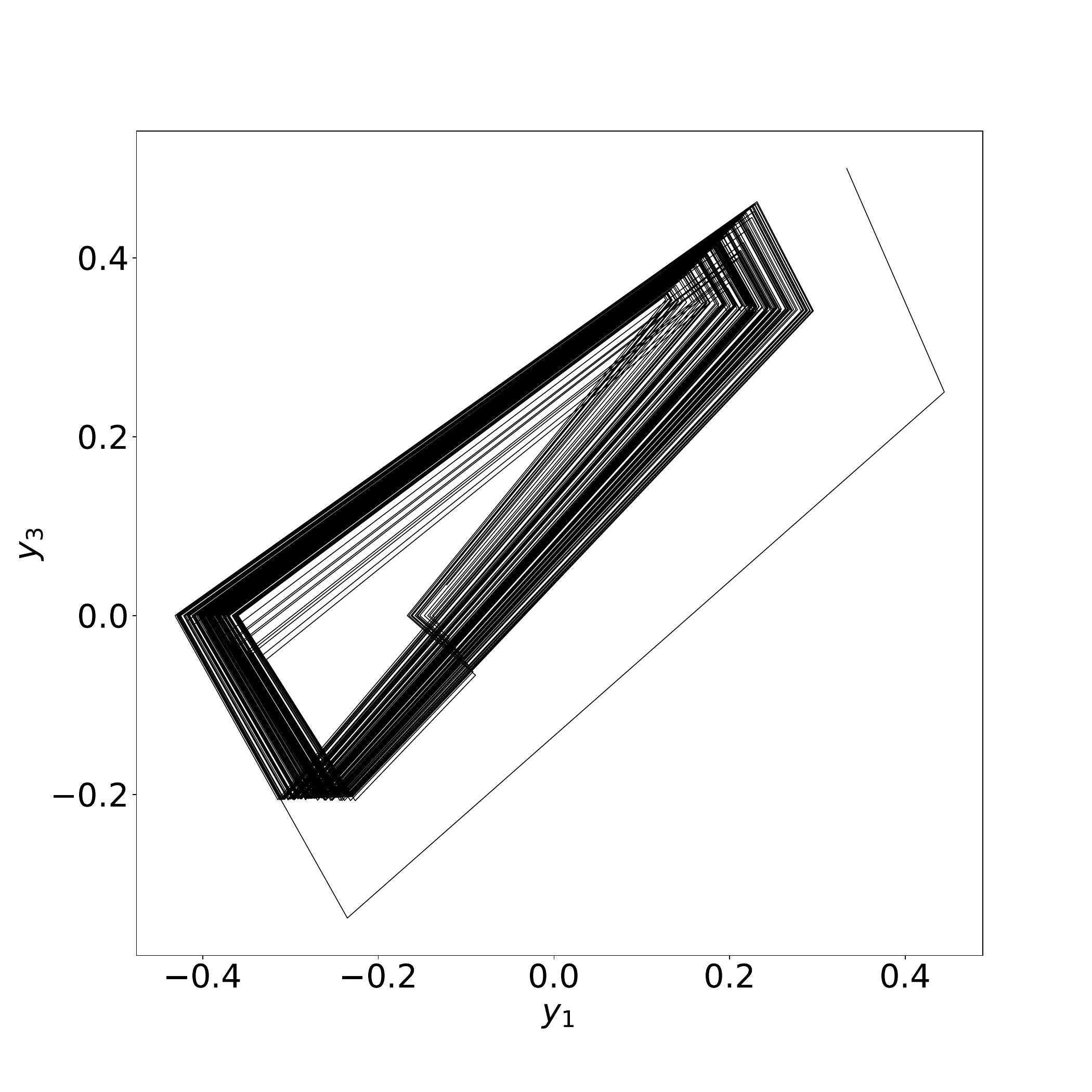}\includegraphics[scale=.18]{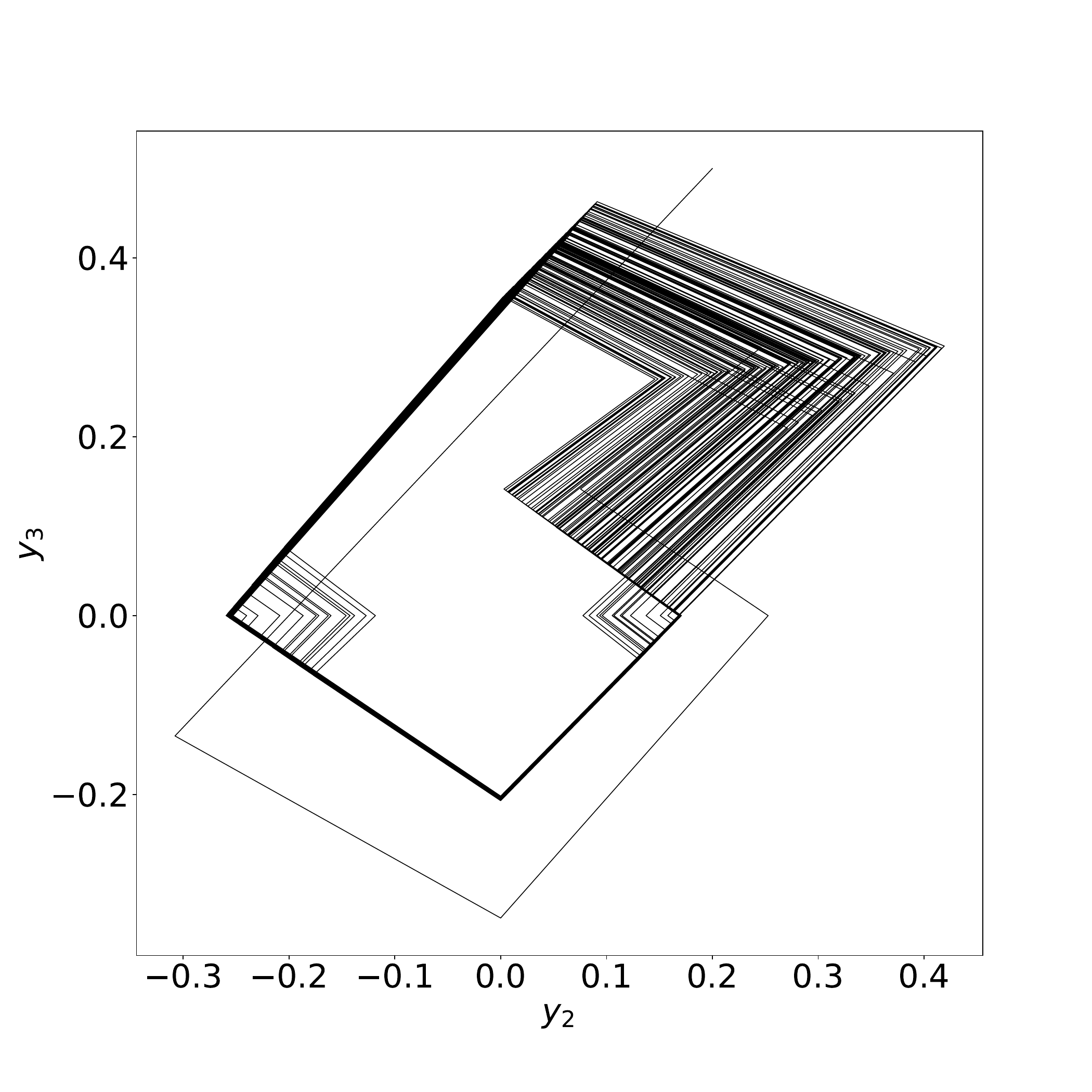}
    \caption{Projections of a phase portrait for example Glass network in Equation~\eqref{eq:example}.}
    \label{GNPP}
\end{figure}

Now, using the wall between boxes $1111$ and $1110$ as a starting wall (which we may denote $+++0$ to indicate the sign of each $y_i$), among the many cycles on the $\TG$ that return to this wall are two, denoted $A$ and $B$, that form a trapping region~\cite{Edwards2001}.
$$
\begin{aligned}
A&:1110\rightarrow1010\rightarrow0010\rightarrow0000\rightarrow0100\rightarrow0110\rightarrow0111\rightarrow1111, \\ 
B&:1110\rightarrow1010\rightarrow0010\rightarrow0011\rightarrow0001\rightarrow0000\rightarrow0100\rightarrow0101\rightarrow0111\rightarrow1111.
\end{aligned}  
$$
Figure~\ref{fig:TGCYCLS} shows the $\TG$ with cycles $A$ and $B$ outlined in red and blue respectively.
\begin{figure}
     \centering
     {\Large
     \begin{subfigure}[b]{0.45\textwidth}
         \centering
        \scalebox{0.38}{
	\begin{tikzpicture}[very thick, every node/.style={sloped,allow upside down},scale=.5]
	\node[inner sep=2pt, circle, draw, fill, label={north west:$0000$}] (0000) at (0,0) {};
    \node[inner sep=2pt, circle, draw, fill, label={north west:$0001$}] (0001) at (8,3) {};
    \node[inner sep=2pt, circle, draw, fill, label={north west:$0010$}] (0010) at (7,9) {};
    \node[inner sep=2pt, circle, draw, fill, label={north east:$0011$}] (0011) at (11,6) {};
    \node[inner sep=2pt, circle, draw, fill, label={north west:$0100$}] (0100) at (0,20) {};
    \node[inner sep=2pt, circle, draw, fill, label={south east:$0101$}] (0101) at (8,11) {};
    \node[inner sep=2pt, circle, draw, fill, label={north west:$0110$}] (0110) at (7,29) {};
    \node[inner sep=2pt, circle, draw, fill, label={north west:$0111$}] (0111) at (11,14) {};
    \node[inner sep=2pt, circle, draw, fill, label={east:$1000$}] (1000) at (20,0) {};
    \node[inner sep=2pt, circle, draw, fill, label={north west:$1001$}] (1001) at (16,3) {};
    \node[inner sep=2pt, circle, draw, fill, label={north west:$1010$}] (1010) at (27,9) {};
    \node[inner sep=2pt, circle, draw, fill, label={north west:$1011$}] (1011) at (19,6) {};
    \node[inner sep=2pt, circle, draw, fill, label={north west:$1100$}] (1100) at (20,20) {};
    \node[inner sep=2pt, circle, draw, fill, label={north west:$1101$}] (1101) at (16,11) {};
    \node[inner sep=2pt, circle, draw, fill, label={north west:$1110$}] (1110) at (27,29) {};
    \node[inner sep=2pt, circle, draw, fill, label={east:$1111$}] (1111) at (19,14) {};
    \draw (0000) --node {\midarrow} (0100)[style=rededge, line width=3pt];
    \draw (0010) --node {\midarrow} (0000)[style=rededge, line width=3pt];
    \draw (1100) --node {\midarrow} (0100);
    \draw (1100) --node {\midarrow} (1000);
    \draw (0110) --node {\midarrow} (0010);
    \draw (0100) --node {\midarrow} (0110)[style=rededge, line width=3pt];
    \draw (0110) --node {\midarrow} (1110);
    \draw (1110) --node {\midarrow} (1010)[style=rededge, line width=3pt];
    \draw (1010) --node {\midarrow} (1000);
    \draw (1000) --node {\midarrow} (0000);
    \draw (1010) --node {\midarrow} (0010)[style=rededge, line width=3pt];
    \draw (1110) --node {\midarrow} (1100);
    \draw (0001) --node {\midarrow} (1001);
    \draw (1001) --node {\midarrow} (1011);
    \draw (1011) --node {\midarrow} (0011);
    \draw (0011) --node {\midarrow} (0001);
    \draw (0101) --node {\midarrow} (0001);
    \draw (1001) --node {\midarrow} (1101);
    \draw (1111) --node {\midarrow} (1011);
    \draw (0011) --node {\midarrow} (0111);
    \draw (0101) --node {\midarrow} (1101);
    \draw (0101) --node {\midarrow} (0111);
    \draw (0111) --node {\midarrow} (1111)[style=rededge, line width=3pt];
    \draw (1101) --node {\midarrow} (1111);
    \draw (0001) --node {\midarrow} (0000);
    \draw (1001) --node {\midarrow} (1000);
    \draw (0010) --node {\midarrow} (0011);
    \draw (1011) --node {\midarrow} (1010);
    \draw (0100) --node {\midarrow} (0101);
    \draw (0110) --node {\midarrow} (0111)[style=rededge, line width=3pt];
    \draw (1111) --node {\midarrow} (1110)[style=rededge, line width=3pt];
    \draw (1100) --node {\midarrow} (1101);
    \end{tikzpicture}
	}
         \caption{$ $}
         \label{fig:aa}
     \end{subfigure}
     \hfill
     \begin{subfigure}[b]{0.45\textwidth}
         \centering
       \scalebox{0.38}{
	\begin{tikzpicture}[very thick, every node/.style={sloped,allow upside down},scale=.5]
	\node[inner sep=2pt, circle, draw, fill, label={north west:$0000$}] (0000) at (0,0) {};
    \node[inner sep=2pt, circle, draw, fill, label={north west:$0001$}] (0001) at (8,3) {};
    \node[inner sep=2pt, circle, draw, fill, label={north west:$0010$}] (0010) at (7,9) {};
    \node[inner sep=2pt, circle, draw, fill, label={north east:$0011$}] (0011) at (11,6) {};
    \node[inner sep=2pt, circle, draw, fill, label={north west:$0100$}] (0100) at (0,20) {};
    \node[inner sep=2pt, circle, draw, fill, label={south east:$0101$}] (0101) at (8,11) {};
    \node[inner sep=2pt, circle, draw, fill, label={north west:$0110$}] (0110) at (7,29) {};
    \node[inner sep=2pt, circle, draw, fill, label={north west:$0111$}] (0111) at (11,14) {};
    \node[inner sep=2pt, circle, draw, fill, label={east:$1000$}] (1000) at (20,0) {};
    \node[inner sep=2pt, circle, draw, fill, label={north west:$1001$}] (1001) at (16,3) {};
    \node[inner sep=2pt, circle, draw, fill, label={north west:$1010$}] (1010) at (27,9) {};
    \node[inner sep=2pt, circle, draw, fill, label={north west:$1011$}] (1011) at (19,6) {};
    \node[inner sep=2pt, circle, draw, fill, label={north west:$1100$}] (1100) at (20,20) {};
    \node[inner sep=2pt, circle, draw, fill, label={north west:$1101$}] (1101) at (16,11) {};
    \node[inner sep=2pt, circle, draw, fill, label={north west:$1110$}] (1110) at (27,29) {};
    \node[inner sep=2pt, circle, draw, fill, label={east:$1111$}] (1111) at (19,14) {};
    \draw (0000) --node {\midarrow} (0100)[style=blueedge, line width=3pt];
    \draw (0010) --node {\midarrow} (0000);
    \draw (1100) --node {\midarrow} (0100);
    \draw (1100) --node {\midarrow} (1000);
    \draw (0110) --node {\midarrow} (0010);
    \draw (0100) --node {\midarrow} (0110);
    \draw (0110) --node {\midarrow} (1110);
    \draw (1110) --node {\midarrow} (1010)[style=blueedge, line width=3pt];
    \draw (1010) --node {\midarrow} (1000);
    \draw (1000) --node {\midarrow} (0000);
    \draw (1010) --node {\midarrow} (0010)[style=blueedge, line width=3pt];
    \draw (1110) --node {\midarrow} (1100);
    \draw (0001) --node {\midarrow} (1001);
    \draw (1001) --node {\midarrow} (1011);
    \draw (1011) --node {\midarrow} (0011);
    \draw (0011) --node {\midarrow} (0001)[style=blueedge, line width=3pt];
    \draw (0101) --node {\midarrow} (0001);
    \draw (1001) --node {\midarrow} (1101);
    \draw (1111) --node {\midarrow} (1011);
    \draw (0011) --node {\midarrow} (0111);
    \draw (0101) --node {\midarrow} (1101);
    \draw (0101) --node {\midarrow} (0111)[style=blueedge, line width=3pt];
    \draw (0111) --node {\midarrow} (1111)[style=blueedge, line width=3pt];
    \draw (1101) --node {\midarrow} (1111);
    \draw (0001) --node {\midarrow} (0000)[style=blueedge, line width=3pt];
    \draw (1001) --node {\midarrow} (1000);
    \draw (0010) --node {\midarrow} (0011)[style=blueedge, line width=3pt];
    \draw (1011) --node {\midarrow} (1010);
    \draw (0100) --node {\midarrow} (0101)[style=blueedge, line width=3pt];
    \draw (0110) --node {\midarrow} (0111);
    \draw (1111) --node {\midarrow} (1110)[style=blueedge, line width=3pt];
    \draw (1100) --node {\midarrow} (1101);
    \end{tikzpicture}
	}
         \caption{$ $}
         \label{fig:bb}
     \end{subfigure}
     }
        \caption{ $\TG$ for example Glass network in Equation~\eqref{eq:example} with cycle $A$ outlined in Red (a) and cycle $B$ outlined in Blue (b).}
   	\label{fig:TGCYCLS}
\end{figure}

As discussed above, to show that the returning cones for these two cycles form a trapping region, one calculates the returning cones for each cycle and shows that their images under their respective cycle maps lie in the union of their returning cones. Thus, any trajectory that follows cycle $A$ or $B$ once will necessarily follow one or the other at each iteration for the rest of time. For cycle $A$ in our example network, from the starting wall there is one alternative exit variable, $i=3$. On the next wall, $+0+-$ (between boxes $1110$ and $1010$), $i=3$ is again the only alternative exit variable. On the next wall, $0-+-$ (between boxes $1010$ and $0010$) the only alternative exit variable is $i=4$. Then, on $--0-$ (between boxes $0010$ and $0000$), there are no alternative exit variables so this wall does not contribute a row to the matrix $R$. On the wall $-0--$ (between boxes $0000$ and $0100$), there is one alternative exit variable, $i=4$. On $-+0-$ (between boxes $0100$ and $0110$), there are two alternative exit variables, $i=1$ and $i=2$. On $-++0$ (between boxes $0110$ and $0111$), there are no alternative exit variables. Finally, on the last wall of the cycle, $0+++$ (between boxes $0111$ and $1111$), there is one alternative exit variable, $i=2$. 

Using all of these alternative exit variables for cycle $A$ in  equation~\eqref{eq:retconerow}, we compute the matrix $R$ for cycle $A$ as
$$
R_A=\begin{pmatrix}
0&-1&1&0\\
-1&-2&1&0\\
2&4&-1&-1\\
2&4&-1&-1\\
-3&-8&4&1\\
-2&-5&2&1\\
-2&-5&2&1
\end{pmatrix}.
$$
Because necessarily $y_4=0$ on the starting wall, we can ignore the final column of $R_A$, and by a slight abuse of notation let $y$ denote just $(y_1,y_2,y_3)^{\top}$. We can also remove duplicate rows of $R_A$, and actually we can remove rows for which the inequality is already implied by another row or rows. Removing all such redundancies, $R_A$ becomes
\begin{equation}
R_A=\begin{pmatrix}
2&4&-1\\
-2&-5&2
\end{pmatrix}.
\end{equation}

Similarly, following the same procedure, the $R$ matrix that defines the returning cone for cycle $B$ is 
\begin{equation}
R_B=\begin{pmatrix}
6&11&-2\\
-2&-4&1
\end{pmatrix}.
\end{equation}
We can also describe the returning cones by means of their extremal vectors (vertices of a polygonal cross section of the cone)~\cite{Edwards2001}. Berman and Plemmons~\cite[pp.1--2]{BermanPlemmons1994} define a cone as follows.
\begin{definition}
For a set $S\subseteq \mathbb R^n$, the set generated by $S$ is the set of finite non-negative linear combinations of elements of $S$:
$$
S^G=\left\{\sum_{i=1}^mc_ix_i,\: \text{for some }c_i\geq0,\;x_i\in S,m\; \text{finite}\right\}.
$$
A set $K$ is a cone if $K=K^G$. 
\end{definition}
So for any $S$, $S^G$ is a cone. The returning cones for cycles $A$ and $B$ are (respectively):
\begin{align}
C_A&=\left\{\left(0,\frac{2}{7},\frac{5}{7}\right)^{\top},\left(\frac{1}{2},0,\frac{1}{2}\right)^{\top},\left(\frac{1}{3},0,\frac{2}{3}\right)^{\top},\left(0,\frac{1}{5},\frac{4}{5}\right)^{\top}\right\}^G,\label{eq:exreturncone1}\\
C_B&=\left\{\left(0,\frac{2}{13},\frac{11}{13}\right)^{\top},\left(\frac{1}{4},0,\frac{3}{4}\right)^{\top},\left(\frac{1}{3},0,\frac{2}{3}\right)^{\top},\left(0,\frac{1}{5},\frac{4}{5}\right)^{\top}\right\}^G. \label{eq:exreturncone2}
\end{align}
On $C_A$ and $C_B$ their respective cycle maps act together as a Poincaré map. 
Since the cycle map acts on (a subset of) the starting wall, which is in $\mathbb{R}^{n-1}$, we can express it in terms of $(n-1)$-vectors and an $(n-1)\times (n-1)$ matrix. 
On the starting wall one of the coordinates is always zero ($y_i=0$), so we can remove it from the cycle map. 
Also, since in the computation of the cycle matrix the $i^{th}$ row is all $0$'s, we can remove that row also. 
So we can reduce the cycle map for a given cycle by one row and column. 
For this example we can write the map~\eqref{eq:cyclemap} for each of the cycles in terms of $3\times 3$ matrices, and $3$-vectors as 
$$ 
M_A=\frac{A{y}}{1+\phi^{\top}{y}},\quad
M_B=\frac{B{y}}{1+\psi^{\top}{y}}\,,
$$
where
$$
A=\begin{pmatrix}
-3&-8&4\\
-2&-5&2\\
-4&-12&7
\end{pmatrix},\;B=\begin{pmatrix}
5&8&0\\
6&11&-2\\
12&20&-1
\end{pmatrix},\\
$$
$$
\phi=(-4,-14,10)^{\top},\;\psi=(12,18,2)^{\top}.
$$
Applying each of the cycle maps to their respective returning cones, one finds that each returning cone gets mapped into the union of the two. Returning cones are proper (non-empty, pointed) cones in ${\mathbb R}^{n-1}$ (a wall in ${\mathbb R}^n$) with vertex at the origin. 

In a Glass network with uniform decay rates, rays map to rays under the mappings from wall to wall and radial dynamics is always convergent (trajectories starting on the same ray converge under iteration of the maps)~\cite{Edwards2000}. So, in order to depict where returning cones, $C_A$ and $C_B$, in our example, are in ${\mathbb R}^3$ (for a Glass network in ${\mathbb R}^4$), and their images under their respective cycle maps, one can represent a point on a ray with the ray's intersection with the plane of unit $L_1$ norm, $y_1+y_2+y_3=1$. 
\begin{figure}
    \centering
    \includegraphics[scale=.3]{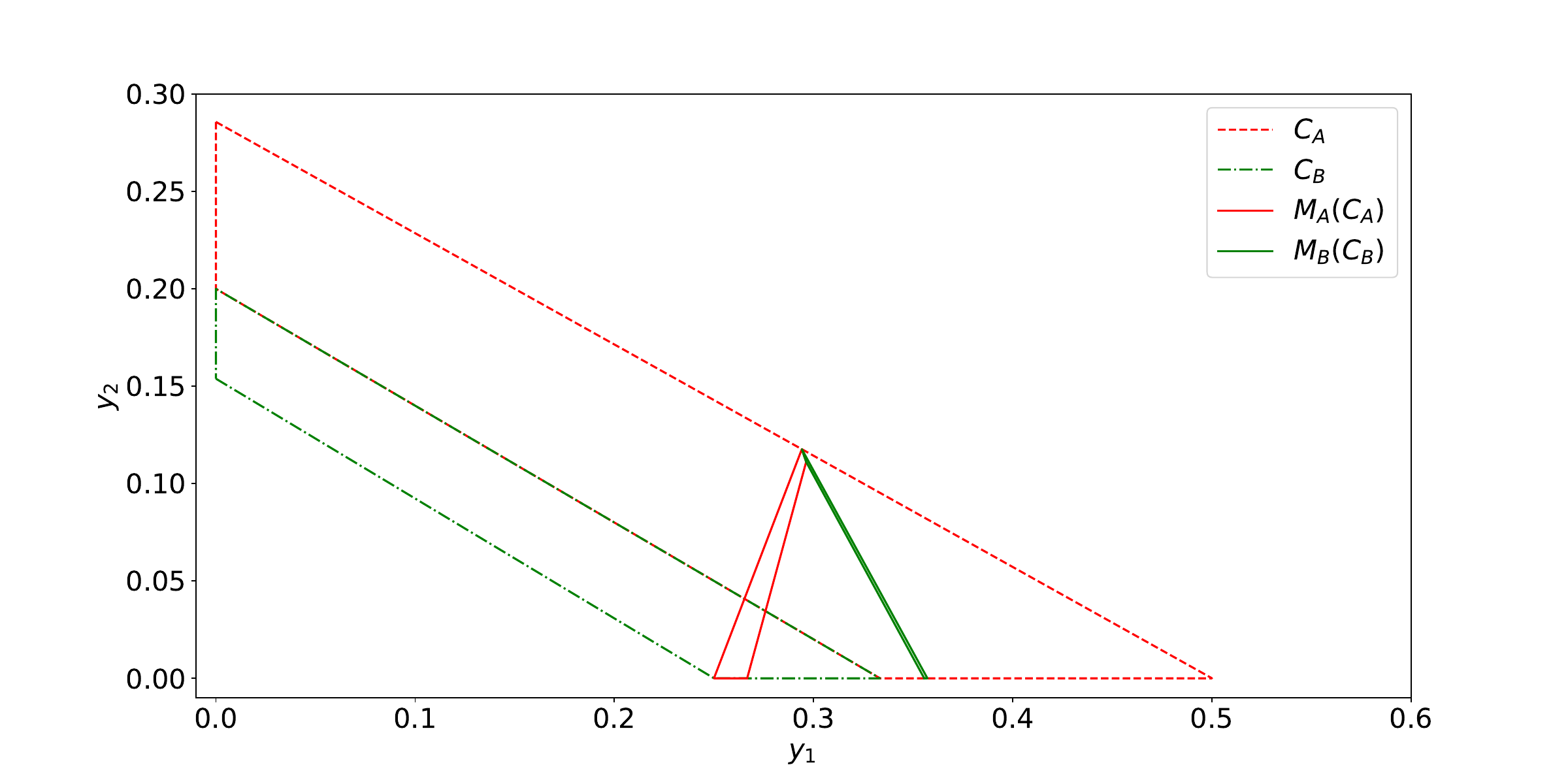}
    \caption{Returning cones for cycles $A$ and $B$ from Equations~\eqref{eq:exreturncone1} and~\eqref{eq:exreturncone2}, represented by the cross section in the plane $y_1+y_2+y_3=1$ with only the first two coordinates plotted, and their images under their respective cycle maps. Note that although $M_B(C_B)$ is very narrow, it has non-empty interior.} 
    \label{fig:RCCS}
\end{figure}
Thus, the 3-dimensional returning cone is represented by a two-dimensional polygon, which is the cross section of the cone with the unit $L_1$-norm plane, with only the first two coordinates plotted in Figure~\ref{fig:RCCS}. The image of each returning cone under its respective map, $M_A(C_A)$ and $M_B(C_B)$ is also depicted in Figure~\ref{fig:RCCS}, showing that all points in $C_A\cup C_B$ map back into $C_A\cup C_B$, which is thus a trapping region. Thus, long term dynamics of this network (or at least the basin of attraction of this trapping region) can be represented on the subset of the $\TG$ that is just the union of cycles $A$ and $B$. Figure~\ref{fig:1TGR} shows the reduced graph, $\TG_r$, that contains only the dynamics allowed by cycles $A$ and $B$. The entropy calculated from the $\TG_r$'s adjacency matrix is $h_{\TG_r}\approx0.224$, about one quarter the entropy of the original $\TG$. However, the $\TG_r$ contains 4 cycles, not just $A$ and $B$, and so includes dynamics not allowed in the trapping region. 

\begin{figure}
	\centering
    {\Large
    \scalebox{0.6}{
	\begin{tikzpicture}[very thick, every node/.style={sloped,allow upside down},scale=.5]
	\node[inner sep=2pt, circle, draw, fill, label={south west:$0000$}] (0000) at (0,0) {};
	\node[inner sep=2pt, circle, draw, fill, label={below:$0001$}] (0001) at (8,3) {};
	\node[inner sep=2pt, circle, draw, fill, label={left:$0010$}] (0010) at (7,9) {};
	\node[inner sep=2pt, circle, draw, fill, label={north east :$0011$}] (0011) at (11,6) {};
	\node[inner sep=2pt, circle, draw, fill, label={west:$0100$}] (0100) at (0,20) {};
	\node[inner sep=2pt, circle, draw, fill, label={west:$0101$}] (0101) at (8,11) {};
	\node[inner sep=2pt, circle, draw, fill, label={west:$0110$}] (0110) at (7,29) {};
	\node[inner sep=2pt, circle, draw, fill, label={north east:$0111$}] (0111) at (11,14) {};
	\node[inner sep=2pt, circle, draw, fill, label={east:$1010$}] (1010) at (27,9) {};
	\node[inner sep=2pt, circle, draw, fill, label={east:$1110$}] (1110) at (27,29) {};
	\node[inner sep=2pt, circle, draw, fill, label={east:$1111$}] (1111) at (19,14) {};
	\draw (0001) --node {\midarrow} (0000);
	\draw (0010) --node {\midarrow} (0000);
	\draw (0000) --node {\midarrow} (0100);
	\draw (0100) --node {\midarrow} (0110);
	\draw (0100) --node {\midarrow} (0101);	
	\draw (0110) --node {\midarrow} (0111);
	\draw (0101) --node {\midarrow} (0111);
	\draw (0111) --node {\midarrow} (1111);
	\draw (0011) --node {\midarrow} (0001);
	\draw (0010) --node {\midarrow} (0011);
	\draw (1010) --node {\midarrow} (0010);
	\draw (1110) --node {\midarrow} (1010);
	\draw (1111) --node {\midarrow} (1110);
	\end{tikzpicture}
    }
    }
	\caption{$\TG_r$ for the example Glass network in Equation~\eqref{eq:example} with only the trapping region.}
   	\label{fig:1TGR}
\end{figure}
\begin{figure}
    \centering
    {\Large
    \scalebox{0.6}{
    \begin{tikzpicture}[very thick, every node/.style={sloped,allow upside down},scale=.5]
    \node[inner sep=2pt, circle, draw, fill, label={north west:$0000_A$}] (0000A) at (0,0) {};
    \node[inner sep=2pt, circle, draw, fill, label={north west:$0010_A$}] (0010A) at (7,9) {};
    \node[inner sep=2pt, circle, draw, fill, label={north west:$1010_A$}] (1010A) at (27,9) {};
    \node[inner sep=2pt, circle, draw, fill, label={north west:$0100_A$}] (0100A) at (0,20) {};
    \node[inner sep=2pt, circle, draw, fill, label={north west:$0110_A$}] (0110A) at (7,29) {};
    \node[inner sep=2pt, circle, draw, fill, label={north east:$1110_A$}] (1110A) at (27,28.5) {};
    \node[inner sep=2pt, circle, draw, fill, label={north west:$0001_B$}] (0001B) at (8,3) {};
    \node[inner sep=2pt, circle, draw, fill, label={east:$0011_B$}] (0011B) at (11,6) {};
    \node[inner sep=2pt, circle, draw, fill, label={east:$0101_B$}] (0101B) at (8,11) {};
    \node[inner sep=2pt, circle, draw, fill, label={north west:$0111_B$}] (0111B) at (11,14) {};
    \node[inner sep=2pt, circle, draw, fill, label={north west:$1111_B$}] (1111B) at (19,14) {};
    \node[inner sep=2pt, circle, draw, fill, label={south east:$0000_B$}] (0000B) at (1.5,0) {};
    \node[inner sep=2pt, circle, draw, fill, label={east:$0100_B$}] (0100B) at (1.5,20) {};
    \node[inner sep=2pt, circle, draw, fill, label={south:$0010_B$}] (0010B) at (8,8) {};
    \node[inner sep=2pt, circle, draw, fill, label={south west:$1010_B$}] (1010B) at (29,8) {};
    \node[inner sep=2pt, circle, draw, fill, label={east:$1110_B$}] (1110B) at (29,34) {};
    \node[inner sep=2pt, circle, draw, fill, label={north west:$1111_A$}] (1111A) at (19,16) {};
    \node[inner sep=2pt, circle, draw, fill, label={north east:$0111_A$}] (0111A) at (11,16) {};
    \draw (0000A) --node {\midarrow} (0100A);
    \draw (0000B) --node {\midarrow} (0100B);
    \draw (0100A) --node {\midarrow} (0110A);
    \draw (0110A) --node {\midarrow} (0111A);
    \draw (0111A) --node {\midarrow} (1111A);
    \draw (1111A) --node {\midarrow} (1110B);
    \draw (1110B) --node {\midarrow} (1010B);
    \draw (1010B) --node {\midarrow} (0010B);
    \draw (0010B) --node {\midarrow} (0011B);
    \draw (0011B) --node {\midarrow} (0001B);
    \draw (0001B) --node {\midarrow} (0000B);
    \draw (0100B) --node {\midarrow} (0101B);
    \draw (0101B) --node {\midarrow} (0111B);
    \draw (0111B) --node {\midarrow} (1111B);
    \draw (1111B) --node {\midarrow} (1110A);
    \draw (1110A) --node {\midarrow} (1010A);
    \draw (1010A) --node {\midarrow} (0010A);
    \draw (0010A) --node {\midarrow} (0000A);
    \draw (1111B) --node {\midarrow} (1110B);
    \draw (1111A) --node {\midarrow} (1110A);
    \end{tikzpicture}
    }
    }
    \caption{$\TG_{r}(1)$ for the example Glass network in Equation~\eqref{eq:example} where cycles $A$ and $B$ have been separated.}
    \label{fig:GNTGSRCD}
\end{figure}

As described in Section~\ref{sec:improved_bound} therefore, we construct a $\TG_r(1)$ in which cycles $A$ and $B$ have been separated from each other. This is shown in Figure~\ref{fig:GNTGSRCD} and its entropy is $h_{\TG_{r}(1)}\approx0.111$. This is less than $h_{\TG_r}$ because the two impossible cycles have been excluded.
\begin{figure}
    \centering
    \includegraphics[scale=.28]{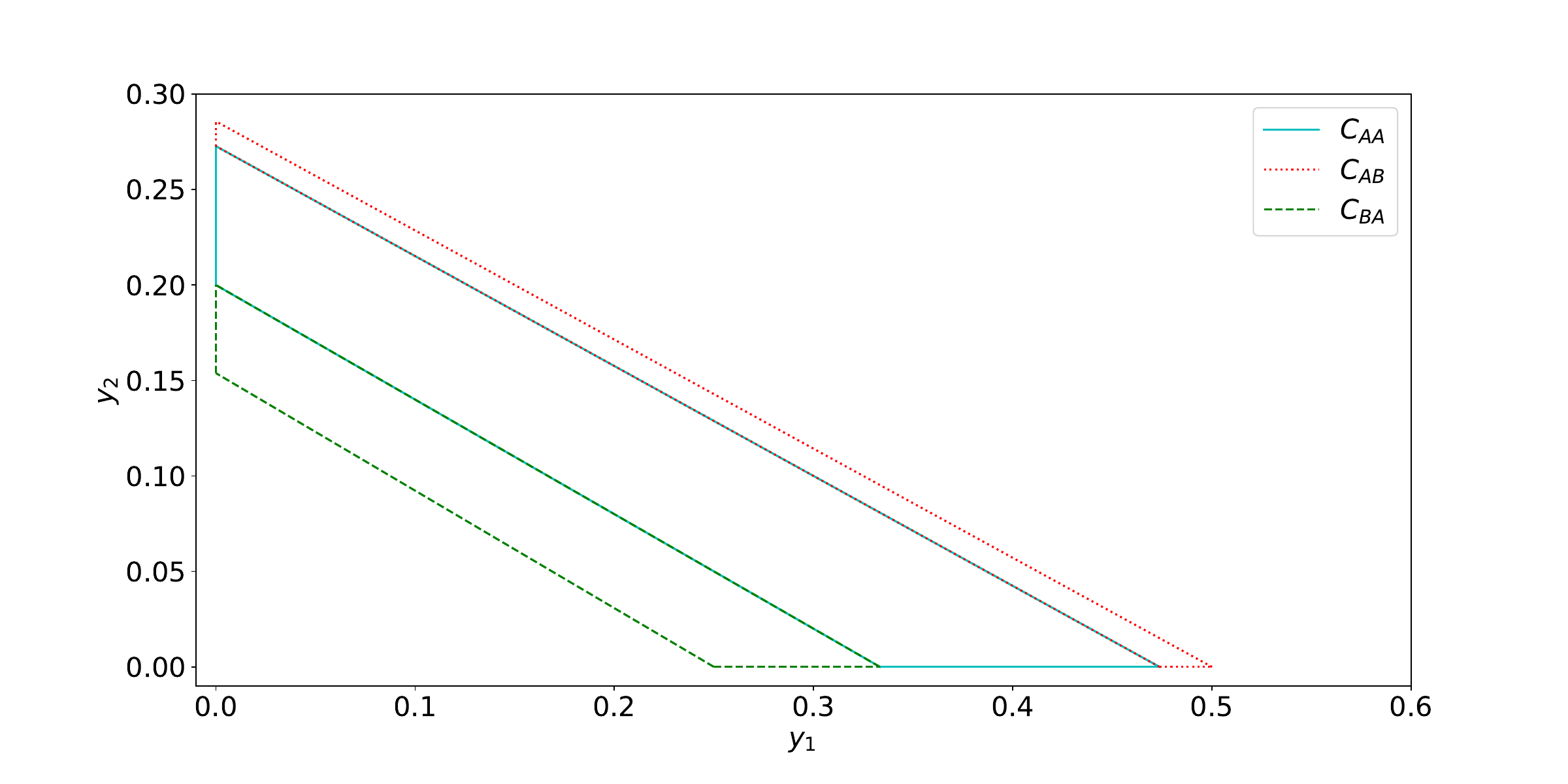}
    \caption{Returning Cones for $AA$, $AB$, $BA$, and $BB$, partitioning the returning cones in Figure~\ref{fig:RCCS}.}
    \label{fig:AAABBARC}
\end{figure}
Additionally, it can be seen in Figure~\ref{fig:RCCS} that the cycle sequence $BB$ is forbidden, since $M_B(C_B)\cap C_B=\emptyset$. In other words, after a circuit of cycle $B$, the next cycle must be $A$, since $M_B(C_B)\subset C_A$. We can also deduce this information by considering the returning cones for each of the length two concatenations. 
These are depicted in Figure~\ref{fig:AAABBARC} and the partitioning of the original returning cones is clear. Concatenation $BB$ is missing from the figure since its returning cone is empty. 
\begin{figure}
    \centering
    {\Large
    \scalebox{0.6}{
    \begin{tikzpicture}[very thick, every node/.style={sloped,allow upside down},scale=.5]
    \node[inner sep=2pt, circle, draw, fill, label={west :$0000_{AB}$}] (0000AB) at (0  , 0) {};
    \node[inner sep=2pt, circle, draw, fill, label={south:$0000_{AA}$}] (0000AA) at (1.5  , 0) {};
    \node[inner sep=2pt, circle, draw, fill, label={south east:$0000_{BA}$}] (0000BA) at (3  , 0) {};
    \node[inner sep=2pt, circle, draw, fill, label={north west:$0100_{AB}$}] (0100AB) at (0, 20) {};
    \node[inner sep=2pt, circle, draw, fill, label={south:$0100_{AA}$}] (0100AA) at (1.5, 20) {};
    \node[inner sep=2pt, circle, draw, fill, label={east:$0100_{BA}$}] (0100BA) at (3, 20) {};
    \node[inner sep=2pt, circle, draw, fill, label={north west:$0110_{AA}$}] (0110AA) at (8, 29) {};
    \node[inner sep=2pt, circle, draw, fill, label={north west:$0110_{AB}$}] (0110AB) at (9, 32) {};
    \node[inner sep=2pt, circle, draw, fill, label={east:$0101_{BA}$}] (0101B) at (9,11) {};
    \node[inner sep=2pt, circle, draw, fill, label={west:$0010_{AA}$}] (0010AA) at (8, 9) {};
    \node[inner sep=2pt, circle, draw, fill, label={south :$0010_{BA}$}] (0010BA) at (9, 8) {};
    \node[inner sep=2pt, circle, draw, fill, label={north west:$0010_{AB}$}] (0010AB) at (7, 10) {};
    \node[inner sep=2pt, circle, draw, fill, label={north west:$0001_{BA}$}] (0001BA) at (10,3) {};
    \node[inner sep=2pt, circle, draw, fill, label={east:$0011_{BA}$}] (0011BA) at (13,6) {};
    \node[inner sep=2pt, circle, draw, fill, label={west:$0111_{AA}$}] (0111AA) at (11,16) {};
    \node[inner sep=2pt, circle, draw, fill, label={north east:$0111_{AB}$}] (0111AB) at (11,17) {};
    \node[inner sep=2pt, circle, draw, fill, label={south east:$0111_{BA}$}] (0111BA) at (11,15) {};
    \node[inner sep=2pt, circle, draw, fill, label={east:$1111_{AA}$}] (1111AA) at (20,16) {};
    \node[inner sep=2pt, circle, draw, fill, label={north west:$1111_{AB}$}] (1111AB) at (20,17) {};
    \node[inner sep=2pt, circle, draw, fill, label={south:$1111_{BA}$}] (1111BA) at (20,15) {};
    \node[inner sep=2pt, circle, draw, fill, label={north west:$1110_{AA}$}] (1110AA) at (28, 28.5) {};
    \node[inner sep=2pt, circle, draw, fill, label={north west:$1110_{BA}$}] (1110BA) at (29, 31) {};
    \node[inner sep=2pt, circle, draw, fill, label={north west:$1110_{AB}$}] (1110AB) at (27, 26) {};
    \node[inner sep=2pt, circle, draw, fill, label={north east:$1010_{AA}$}] (1010AA) at (28, 9) {};
    \node[inner sep=2pt, circle, draw, fill, label={north west:$1010_{AB}$}] (1010AB) at (27, 10) {};
    \node[inner sep=2pt, circle, draw, fill, label={south :$1010_{BA}$}] (1010BA) at (29, 8) {};
    \draw (0000AB) --node {\midarrow} (0100AB);
    \draw (0000AA) --node {\midarrow} (0100AA);
    \draw (0000BA) --node {\midarrow} (0100BA);
    \draw (0100AB) --node {\midarrow} (0110AB);
    \draw (0100AA) --node {\midarrow} (0110AA);
    \draw (0100BA) --node {\midarrow} (0101B);
    \draw (0101B) --node {\midarrow} (0111BA);
    \draw (0110AB) --node {\midarrow} (0111AB);
    \draw (0110AA) --node {\midarrow} (0111AA);
    \draw (0111BA) --node {\midarrow} (1111BA);
    \draw (0111AB) --node {\midarrow} (1111AB);
    \draw (0111AA) --node {\midarrow} (1111AA);
    \draw (0010AA) --node {\midarrow} (0000AA);
    \draw (0010AB) --node {\midarrow} (0000AB);
    \draw (0010BA) --node {\midarrow} (0011BA);
    \draw (0011BA) --node {\midarrow} (0001BA);
    \draw (0001BA) --node {\midarrow} (0000BA);
    \draw (1110AB) --node {\midarrow} (1010AB);
    \draw (1110AA) --node {\midarrow} (1010AA);
    \draw (1110BA) --node {\midarrow} (1010BA);
    \draw (1010BA) --node {\midarrow} (0010BA);
    \draw (1010AA) --node {\midarrow} (0010AA);
    \draw (1010AB) --node {\midarrow} (0010AB);
    \draw (1111AB) --node {\midarrow} (1110BA);
    \draw (1111AA) --node {\midarrow} (1110AA);
    \draw (1111BA) --node {\midarrow} (1110AB);
    \draw (1111BA) --node {\midarrow} (1110AA);
    \draw (1111AA) --node {\midarrow} (1110AB);
    \end{tikzpicture}
    }
    }
    \caption{$\TG_{r}(2)$ for the example Glass network in Equation~\eqref{eq:example}. }
    \label{fig:TGSR2}
\end{figure}
We can then create copies of cycle $A$ and $B$ that correspond to the three concatenations with non-empty returning cones. The cycle left out is of course $BB$. Constructing the $\TG_r(2)$ as described previously gives the graph in Figure~\ref{fig:TGSR2}. The entropy of this graph is $h\approx 0.0813$. 

It is clear from entropy estimates provided by the graph representations in Figures~\ref{fig:1TGR}, \ref{fig:GNTGSRCD}, and~\ref{fig:TGSR2} that our graph refinements quickly improve upon the original entropy bound of Farcot~\cite{Farcot2006}, without much extra work. An alternative justification involving state-splitting for the cycle separation procedure is given in the Appendix.

\section{Numerical Estimation} \label{sec:numerics}
Here, we numerically simulate our example network and extract the number of blocks from long trajectories to get estimates of the entropy as a check on the results of our refinements above. Numerical integration is done here simply by computing the wall-to-wall maps as a discrete process from a given initial point on the starting wall.

\subsection{Have we found all of the blocks?}
It is reasonable to expect that if we simulate many trajectories from different random initial conditions, we may be able to generate all the elements of $\mathcal{B}_n(\overline{\phi(\mathscr{D}_{TR})})$ for reasonably large $n$. This is easy to verify for small $n$ since it is simple to calculate all the returning regions and the necessary trajectories are short. 

We experimented with blocks of length $n = 50$, where $10^5$ steps (wall-to-wall transitions) seemed to generate most of the blocks (the count appeared to have stopped increasing), but a longer simulation showed that after about $10^8$ steps there was another jump in the number of blocks. This is likely due to a few blocks of length $50$ having very narrow returning cones and thus not occurring often. Increasing the trajectory length did not cause any more increases up to $10^9$ steps. Of course, we do not know for certain if there are additional blocks of length $50$ or not, without calculating the returning regions of each possible block of length $50$ and checking to see which are empty, but for large $n$ that exhaustive check is computationally prohibitive. This experiment suggested that $10^9$ steps might be sufficient to get a good estimate of entropy for $n$ somewhat larger than $50$. 

\begin{figure}
    \centering
    \includegraphics[scale = 0.32]{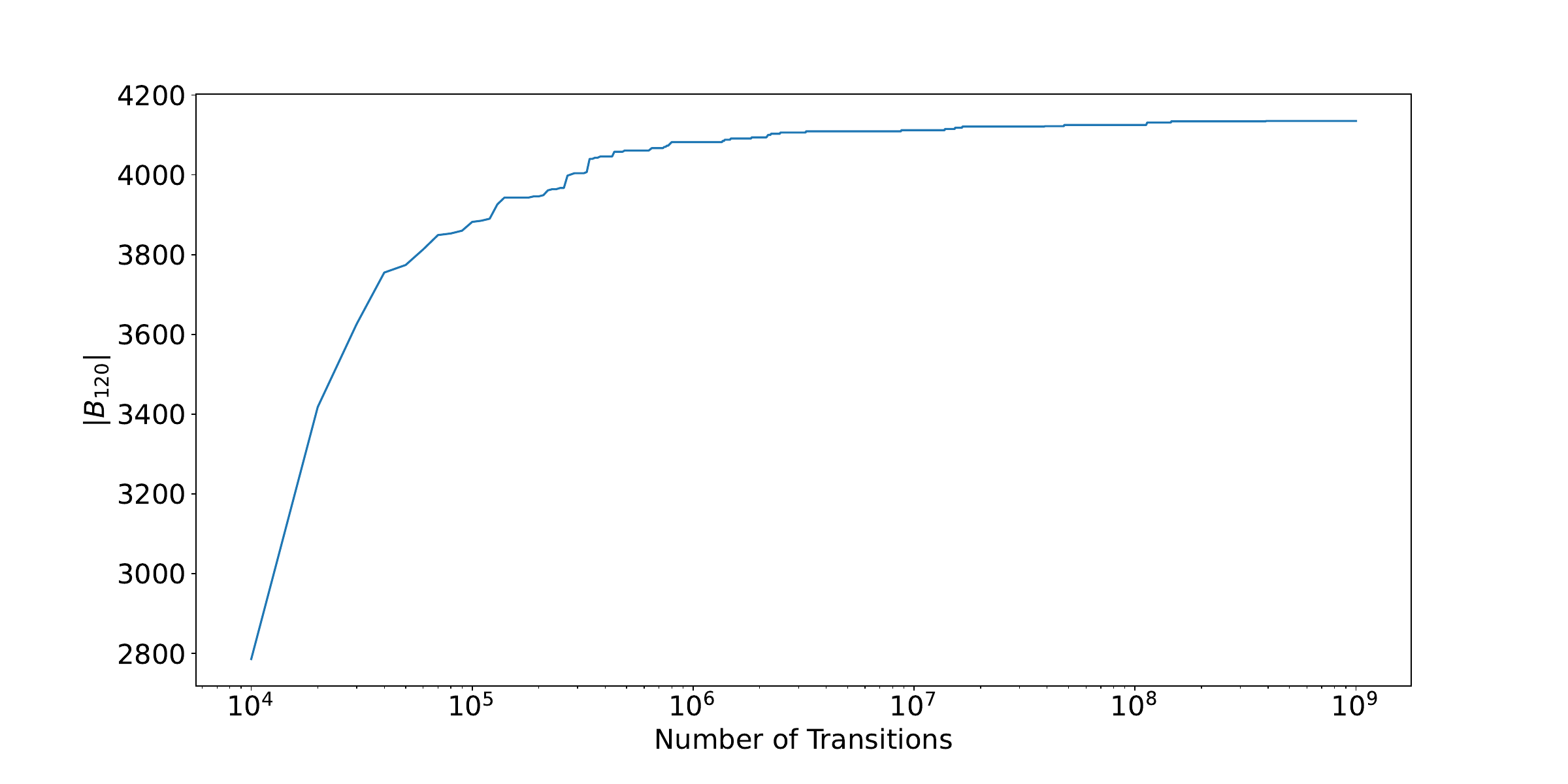}
    \caption{Blocks of length 120 vs number of transitions up to $10^9$}
    \label{fig:B120vstrajlng}
\end{figure}

In Figure~\ref{fig:B120vstrajlng} we plot the number of blocks of length $120$ found in simulations of length up to $10^9$ transitions (with the number of transitions plotted on a logarithmic scale). It is clear that the number of blocks continues to increase until about $10^8$ steps. There are additional small jumps before about $1.5\times 10^8$, but no further visible increases from there until $10^9$ steps. 
If additional increases occur for larger $n$, their impact on entropy should be small, since that involves taking the logarithm of the number of blocks and dividing by $n$. Thus, we expect that for $n\leq 120$, $10^9$ transitions will give us a reasonably tight lower bound on the number of blocks. If the system is chaotic, then the number of blocks of length $n$ continues to increase with $n$, and we can only estimate the number of blocks up to some finite $n$. The count of blocks of some particular (large) length $n$ from a long simulation gives a lower bound on $\mathcal{B}_n(\overline{\phi(\mathscr{D}_{TR})})$ and hence on $\log(\mathcal{B}_n(\overline{\phi(\mathscr{D}_{TR})}))$ for that particular $n$ and for any larger values of $n$, but not necessarily on the entropy, since $\frac 1n \log(\mathcal{B}_n(\overline{\phi(\mathscr{D}_{TR})}))$ may still decrease as $n$ increases as the effects of possible longer forbidden blocks become significant. Thus, our lower bound on the number of blocks for a specific large $n$ does not give a rigorous lower bound for the entropy.

One question that should be addressed is ``what if this is just a very complicated limit cycle?''. For our example it has been proven that there is no stable limit cycle~\cite{Edwards2001}. However, in general this is something that needs to be considered. It has been shown that in example networks with only $4$ variables, there can be surprisingly long stable limit cycles: examples with stable limit cycles of length 174 and 252 transitions have been identified~\cite{Edwards2000}. Without knowing of the existence of such long stable limit cycles ahead of time, numerical simulations would need to be long enough to identify that the number of blocks stops increasing at the length of the cycle. However, our method of upper bounds would eventually catch this, if refinement was carried far enough. If there exists a stable limit cycle involving multiple returns to a starting wall, the graph would eventually reduce to a single long loop without any branching, and which crossed the starting wall multiple times. This structure always has an entropy of 0 and hence would make numerical simulation irrelevant.

\subsection{Numerical (non-rigorous) bound on entropy} 
For a shift space $X$ with nonzero finite entropy, $\mathcal{B}_n(X)$ grows approximately exponentially. For our example, it is reasonable to assume that for sufficiently large $n$, $|\mathcal{B}_n(\overline{\phi(\mathscr{D}_{TR})})|\approx a\cdot b^{n}$, where $a$ and $b$ are positive real constants. Hence,  $h_{\phi(\mathscr{D}_{TR})} \approx \log b $. Thus for chaotic systems, a plot of $\log |\mathcal{B}_n(\overline{\phi(\mathscr{D}_{TR})})|$ vs $n$ should have a linear trend, at least asymptotically. 

\begin{figure}
    \centering
    \includegraphics[scale = 0.32]{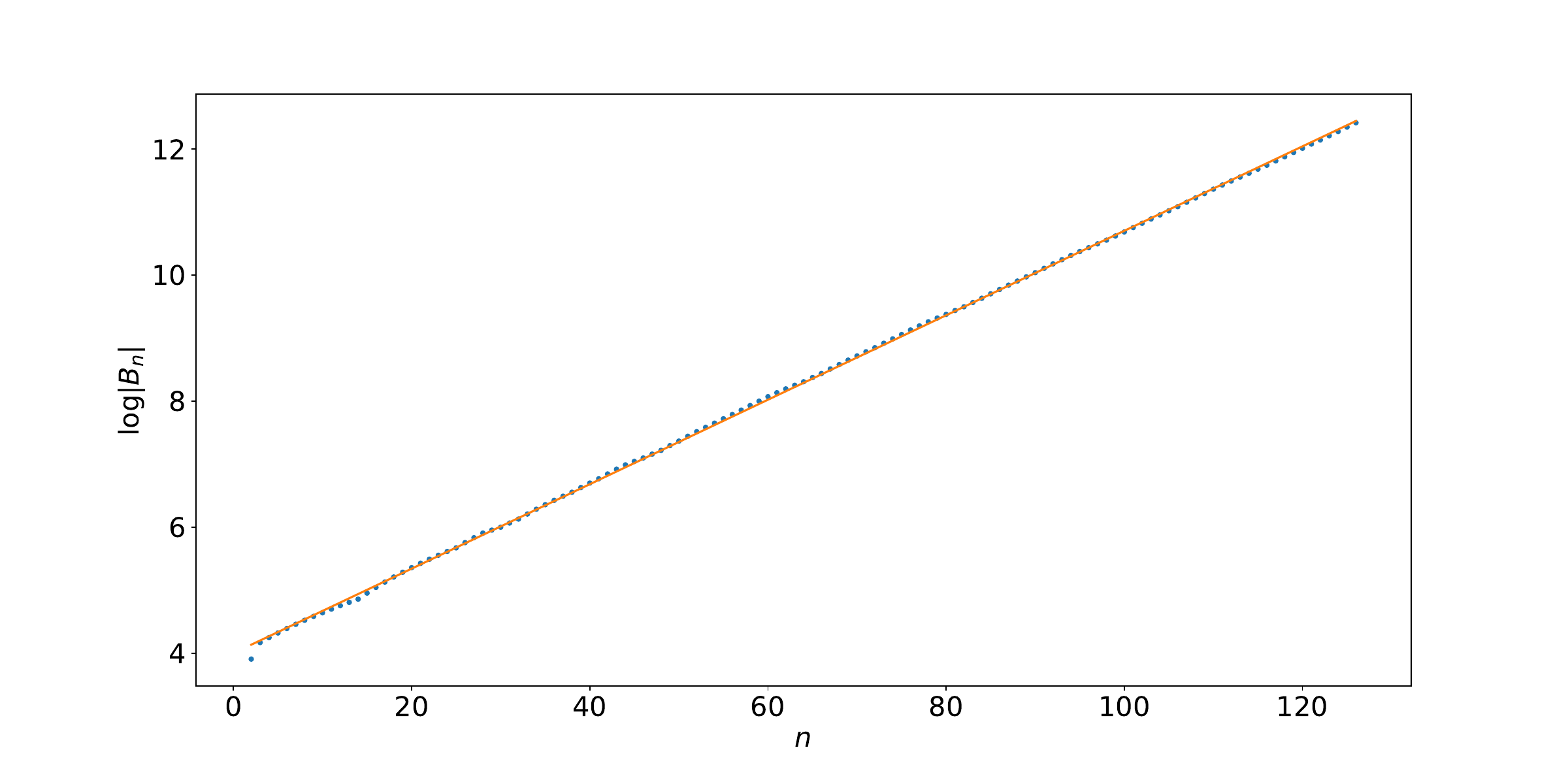}
    \caption{Logarithm (base 2) of number of blocks of length $n$ for $2\le n\le 126$ using a simulation of $10^9 $ transitions (dots) at each $n$. The solid line is the least squares best fit, which has slope $\approx 0.0670258$.}
    \label{fig:Eninterpplot}
\end{figure}

Figure~\ref{fig:Eninterpplot} plots $\log|\mathcal{B}_n(\overline{\phi(\mathscr{D}_{TR})})|$ against $n$ for our example, calculated from numerically generated trajectories. The slope of the best fit line (least squares) is $0.067025$, which gives an estimate of the entropy. Our second refinement from Figure~\ref{fig:TGSR2} has an entropy of approximately $0.081$. It is likely that there is a longer forbidden block that would only be found by refining further, so that $0.081$ is an over-estimate. Alternatively, it may be that there is a very small returning cone that is visited extremely rarely, and was missed by the numerical simulation. If this is the case, then the numerical estimate of $0.067$ underestimates entropy. It may also be that transients occur in the numerical simulations. In the theoretical (upper-bound) estimate, transients of some length are discovered as slightly longer sequences of cycles with an empty returning cone (see above). 
But if the numerical simulations include a transient, they will inflate the number of blocks, and potentially lead to an over-estimate of entropy. However, for a given trajectory, there can be at most one transient block of any given length $n$, so the effect on entropy is negligible for large $n$.

An observation from the numerical simulations is that the word $BAAB$ does not appear in any of the generated trajectories. However, its returning cone is nonempty. It may be that $BAAB$ is a rare sequence. The returning cone is very narrow and on the edge of the trapping region. Furthermore, the entropy of a representation that forbids $BAAB$ is $0.0706$, close to our numerical estimate of $0.067$, so our numerical simulation may have missed a rare occurrence and the true entropy may be closer to $0.081$. On the other hand, it may also be that every word of some length $n>4$ that includes $BAAB$ is forbidden. Given that we know this system is chaotic (or at least, aperiodic), the upper bound of $0.081$ may be sufficient. While $0.067$ is not a rigorous lower bound, it seems likely that the actual entropy is larger than this. The numerical estimate, $0.067$ and the theoretical upper bound, $0.081$, are significantly closer in value than any of the first three estimates of $0.873$, $0.224$, and $0.111$ from the $\TG$, $\TG_r$ and $\TG_{r}(1)$ respectively. Additionally, the upper bound $0.081$ was achieved with very little work. In general, the refinement process may require more effort, but the process is simple to implement.

\section{Conclusions}

We have shown how to improve on the result of Farcot~\cite{Farcot2006}, which allows for the entropy of the dynamics of a Glass network to be bounded above by that of a symbolic dynamical system based on the $\TG$. If one uses more information about the dynamics of such a network, by means of returning regions of cycles and trapping regions, one can construct discrete representations that more faithfully represent the dynamical possibilities, and thus give a tighter upper bound on entropy. The method uses structural changes to separate cycles in the $\TG$ and remove cycles that correspond to unrealizable trajectories, thus reducing the estimated entropy. The procedure can be taken to an arbitrary level of precision, giving a sequence of shift spaces that have decreasing entropy, and which we show approach the true entropy in the limit. 

It should be noted that one could, in principle, avoid our cycle-separation procedure, and simply identify forbidden sequences in the original alphabet that correspond to forbidden sequences of cycles, or even sequences of boxes that do not necessarily correspond to full cycles. This would avoid the need for defining larger and larger alphabets on the refined graphs. However, the information about the dynamics we have comes from returning regions for cycles and trapping regions on a starting wall, and this is the information we use to identify forbidden cycles or sequences of cycles, so it is natural to structure our symbolic dynamics around them. Additionally, the refined graphs have the advantage of allowing simple calculation of the entropy via the Perron eigenvalue of the graph's adjacency matrix. 

We have removed from consideration certain types of network structure that would make our procedure more difficult, in particular, networks in which one cannot avoid an infinite number of first-return cycles on any starting wall, or at least a potentially infinite number based on the $\TG$ alone (Condition~\ref{cond5}). 
It might be possible to extend our work to include such examples. However, the remaining class of networks is large and dynamically diverse, with many good candidates for designs of TRNGs. 

Indeed, one potential application of this work is to quantify entropy in free-running electronic circuits based on standard Boolean logic gates, which can be modelled as Glass networks and could serve as designs for TRNGs. Randomness in a TRNG comes mainly from thermal noise, but the strength of the idea of an underlying chaotic (deterministic) circuit design is that there is already positive entropy even without the thermal noise. The method proposed here allows this entropy to be estimated.

We have assumed throughout that we have a minimal trapping region. However, if there are multiple trapping regions in a given starting wall, we must deal with each minimal trapping region (and thus each attractor) separately in order to use the procedure detailed in Section~\ref{sec:improved_bound}.

It should be possible to automate our method for estimating entropy by an upper bound based on dynamical information that allows refinement of the state transition graph. Given a network structure, and no prior information about the dynamics, can one automatically detect a trapping region and extract a good upper bound on entropy? This is an idea for future work.

\section*{Acknowledgments} This work was partially supported by a Discovery Grant to RE from the Natural Sciences and Engineering Research Council (NSERC) of Canada, and a British Columbia Graduate Scholarship to BW. 

\section*{References}
\bibliographystyle{siamplain}	
\bibliography{References.bib}

\end{document}